%% file: SkeinSurfaceCercle.tex
\documentclass[11pt]{amsart}

\usepackage{amsfonts, amstext, amsmath, amsthm, amscd, amssymb}
\usepackage{epsfig, graphics, psfrag}
\usepackage{color}
\usepackage{url}
\usepackage[margin=1.2in]{geometry}
\usepackage{adjustbox}

\let\oldmarginpar\marginpar
\renewcommand\marginpar[1]{\oldmarginpar[\raggedleft\footnotesize #1]%
{\raggedright\footnotesize #1}}

\newcounter{notes}
\renewcommand{\setminus}{{\smallsetminus}}

\newcommand{\Z}{{\mathbb{Z}}}

\newcommand{\C}{{\mathbb{C}}}
\newcommand{\QQ}{{\mathbb{Q}}}

\newcommand{\Sk}{\mathcal{S}}
\newcommand{\MCG}{\mathrm{Mod}}

\newcommand{\id}{\mathrm{id}}

\theoremstyle{plain}
\newtheorem{theorem}{Theorem}[section]
\newtheorem{corollary}[theorem]{Corollary}
\newtheorem{lemma}[theorem]{Lemma}
\newtheorem{proposition}[theorem]{Proposition}

\newtheorem{conjecture}[theorem]{Conjecture}

\newtheorem*{namedtheorem}{\theoremname}
\newcommand{\theoremname}{testing}

\theoremstyle{definition}
\newtheorem{definition}[theorem]{Definition}

\title[A Basis of the skein module of $\Sigma \times S^1$]{A basis for the
Kauffman skein module of the product of a surface and a circle}

\author{Renaud Detcherry}
\address{Max Planck Institute for Mathematics \\
         Vivatsgasse 7, 53111 Bonn, Germany \newline
         {\tt \url{http://people.mpim-bonn.mpg.de/detcherry}}
         }
\email{detcherry@mpim-bonn.mpg.de}
\author{Maxime Wolff}
\address{Universit\'e Pierre et Marie Curie - Paris 6
\\ Institut de Math\'ematiques de Jussieu
\\ 4 place Jussieu 75005 Paris
\\ {\tt \url{https://webusers.imj-prg.fr/~maxime.wolff/}}}
\email{maxime.wolff@imj-prg.fr}

\begin{document}

\begin{abstract}
  The Kauffman bracket skein module $\mathcal{S}(M)$ of a $3$-manifold~$M$
  is a $\mathbb{Q}(A)$-vector space spanned by links in~$M$ modulo the
  so-called Kauffman relations.
  In this article, for any closed oriented surface~$\Sigma$ we provide an
  explicit spanning family for the skein modules $\mathcal{S}(\Sigma \times S^1).$ 
  Combined with earlier work of Gilmer and Masbaum~\cite{GM18}, we answer
  their question about the dimension of $\mathcal{S}(\Sigma \times S^1)$
  being $2^{2g+1}+2g-1.$
\end{abstract}


\maketitle

\section{Introduction}
\label{sec:intro}
The Kauffman bracket~\cite{Kau} skein module is an invariant of compact
oriented $3$-manifolds. It was first introduced independently by
Przytycki~\cite{Prz} and Turaev~\cite{Tur} as a way to generalize the Jones
polynomial of links in~$S^3.$ It can be thought as a module over any ring~$R$
containing an invertible element $A \in R.$
The skein module $\Sk(M,R)$ with coefficients in~$R$ is the $R$-module:
\[ \Sk(M,R)= \mathrm{Span}_R(L \subset M \ \textrm{framed link})/_{\textrm{isotopy}, K_1,K_2} \]
spanned by isotopy classes of framed links in~$M,$ modulo the two (local)
Kauffman relations $K_1$ and $K_2:$
\begin{center}
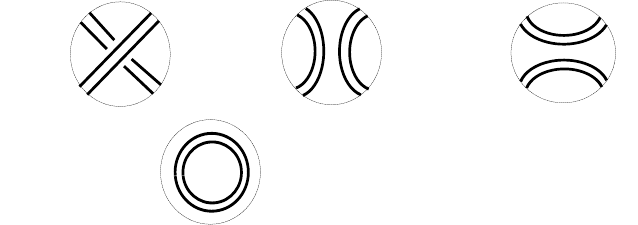
\end{center}
where the above relations identify linear combinations of framed links that
are identical except in a small ball in~$M.$
In the significant case where $R=\QQ(A)$ is the field
of rational functions in the variable~$A,$ we will write $\Sk(M)$
for $\Sk(M,\QQ(A))$ for simplicity.

The skein modules of $3$-manifolds have long been mysterious and notoriously
hard to compute, and for a long time, very little
was known about the structure of skein modules of general $3$-manifolds.
They were partially or completely computed for an increasing
number of closed $3$-manifolds ($S^3$ and lens spaces by Hoste and
Przytycki~\cite{HP93}\cite{HP95}, integer Dehn surgeries on the trefoil by
Bullock~\cite{Bul}, some prism manifolds by Mroczkowski~\cite{Mro11}, the
quaternionic manifold by Gilmer and Harris~\cite{GH07}, the $3$-torus by
Carrega~\cite{Car17} and Gilmer~\cite{Gil18}, and some infinite family of
hyperbolic manifolds by the
first author~\cite{Det19}). Finally, Witten conjectured that $\Sk(M)$ has finite
dimension for every closed $3$-manifold
(see~\cite[Section~8]{GM18} for a discussion).
Recently, Gunningham, Jordan and Safronov posted a general proof~\cite{GJS19}
of Witten's finiteness conjecture. Their proof, which relies on factorization
algebras, the theory of $DQ$-modules and some careful quantization of character
varieties/character stacks, is powerful and generalizes
to other types of skein modules than the Kauffman bracket skein module.

However, the proof in \cite{GJS19} is rather non-constructive; in
particular, it is still hard for a given $3$-manifold~$M$ to compute the
dimension $\mathrm{dim}_{\QQ(A)}(\Sk(M))$ or to find a family of framed links
that is a basis of $\Sk(M).$ It would be rather interesting to give some
general interpretation of the dimensions $\mathrm{dim}_{\QQ(A)}(\Sk(M)).$
Another aspect of skein modules that remains unclear is the integral structure
of skein modules of closed $3$-manifolds. Indeed, a more precise statement,
concerning the ``integral'' version of the skein module $\Sk(M),$ that is,
the skein module $\Sk(M,\Z[A^{\pm 1}])$ with coefficients in $\Z[A^{\pm 1}],$
has been conjectured by Julien March\'e.
For $n\in \Z$ let $\lbrace n \rbrace=A^n-A^{-n}.$
\begin{conjecture}\label{conj:intFinConj}
  Let $M$ be a closed compact oriented $3$-manifold. There exists an integer
  $d\geqslant 0$ and finitely generated $\Z[A^{\pm 1}]$-modules $N_k$ so that
  \[ \Sk(M,\Z[A^{\pm 1}])=\Z[A^{\pm 1}]^d \oplus \underset{k \geqslant 1}{\bigoplus} N_k \]
  where, furthermore, the module $N_k$ is a $\lbrace k \rbrace$-torsion module.
\end{conjecture}
The two authors heard of this conjecture from Julien March\'e by private communication.
To the best of our knowledge, this integral version of the finiteness
conjecture is not implied by Gunningham, Jordan and Safronov's results.
The integral structure of skein modules is also of interest because of its
interaction with quantum invariants, as we will explain later in this introduction.

In the whole article, we focus on a single family of
manifolds: products $\Sigma \times S^1$ of a compact closed oriented surface
$\Sigma$ of genus $g\geqslant 2$ and a circle. (Although,
we believe that our techniques can be employed for other circle bundles
over a closed surface).
We note that skein elements in $\Sigma \times S^1$ admit some particularly
nice diagrammatic representations as so-called \textit{arrowed diagrams}.
Indeed, we can always put a link $L \subset \Sigma \times S^1$ in general
position with respect to the projection $\Sigma \times S^1 \rightarrow \Sigma,$
and besides the over/under-crossing information, we only need to remember
where the link~$L$ intersects $\Sigma \times \lbrace 0 \rbrace.$
We remember this information by putting an arrow on the diagram at such
intersections; the direction of the arrow gives the direction in which~$L$ is rising.
This notation was first introduced by Dabkowski and
Mroczkowski in~\cite{DM09}, where it was used to compute the skein module
of $\Sigma_{0,3}\times S^1.$

With this notation, let us introduce a family $\mathcal{B}$
of elements of $\Sk(\Sigma\times S^1)$ consisting of the following diagrams.
For each non-zero homology class of $H_1(\Sigma,\Z/2)$, we choose an oriented,
simple closed curve representing it,
and consider the diagram consisting of this curve with no arrows, as well as the
diagram consisting of this representant with~$1$ arrow. Finally, we consider
a trivial curve on~$\Sigma$, with~$0$ to $2g$ arrows. Then, our main
result is the following.
\begin{theorem}\label{thm:basis_skein}
  The family $\mathcal{B}$ is a basis of the skein module $\Sk(\Sigma\times S^1)$.
  In particular, $\Sk (\Sigma \times S^1)$ has dimension $2^{2g+1}+2g-1.$
\end{theorem}
Our proof is completely elementary; it uses only skein
relations and direct computation. It is independent of~\cite{GJS19}.

The skein modules of the manifolds $\Sigma \times S^1$ were previously studied
by Gilmer and Masbaum~\cite{GM18}.
They introduced a general tool to bound below the dimension of skein modules.
Given a closed compact oriented $3$-manifold~$M,$ Gilmer and Masbaum defined a
linear map
\[ \mathrm{ev} \colon \Sk(M) \rightarrow \C^{\mathcal{U}}_{ae}. \]
In the above, $\mathcal{U}$
denotes the set of roots of unity of even order, and
$\C^{\mathcal{U}}_{ae}$ is the $\QQ(A)$-vector space of functions
of the variable
$A \in \mathcal{U}$ that are defined almost everywhere. For example, any
function $F(A) \in \QQ(A)$ is an element of $\C^{\mathcal{U}}_{ae}$ as it is
defined as long as~$A$ is not a root of the denominator of~$F.$ Moreover, for any
$k \in \Z,$ there is a map $p^k\in \C^{\mathcal{U}}_{ae}$ which maps any
primitive root of unity of order~$2p$ to~$p^k.$
The map $\mathrm{ev}$ is defined using the Reshetikhin-Turaev invariants of
links in $M:$
given a $2p$-th root of unity $A$
where $p\geqslant 3,$ and a link $K \subset M,$
there is a well-defined
topological invariant $RT(M,K,A)\in \C.$ The invariant can be computed by
choosing any surgery presentation~$L$ for~$M,$ and computing a colored
Kauffman bracket of $L \cup K,$ where~$L$ has been colored by some special
polynomial $\omega_p \in \Z[A^{\pm 1}][z]$ called the Kirby color.
We will not give a complete definition of $RT(M,K,A),$ and just
refer to~\cite{GM18} for details. The important point to us is that, with
this definition, the invariant $RT(M,K,A)$ naturally satisfies the Kauffman
relations with respect to $K.$ Thus it is possible to extend it
$\QQ(A)$-linearly from the set of framed links to any element of $\Sk(M),$
as those elements are $\QQ(A)$ linear combinations of links in $M.$
The only caveat is that we may have to exclude some values of
$A \in \mathcal{U}$ if they are in denominators of the coefficients of the
linear combination.
Using the map $\mathrm{ev}$ in the special case where $M=\Sigma \times S^1,$
Gilmer and Masbaum
showed that the dimension of its image is at least $2^{2g+1}+2g-1,$
and thus so is the dimension of $\Sk(\Sigma \times S^1).$
They conjectured that this is actually the dimension of $\Sk(\Sigma \times S^1).$
Theorem~\ref{thm:basis_skein} answers positively to their conjecture;
our contribution is to prove that the family $\mathcal{B}$, of cardinal
$2^{2g+1}+2g-1$, generates $\Sk(\Sigma\times S^1).$
Our proof is constructive: given a link~$L$ in $\Sigma \times S^1,$ it is
possible to algorithmically decompose it as a linear combination of
elements of~$\mathcal{B}$.

For simplicity we stated Theorem~\ref{thm:basis_skein} with the skein
module with coefficients in $\QQ(A)$, but in fact we only need
$\lbrace k\rbrace=A^k-A^{-k}$ to be invertible, for all $k\neq 0$.
In particular, a by-product of our proof is that torsion elements
of the integral skein module $\Sk(\Sigma \times S^1,\Z[A^{\pm 1}])$ are always of
$\lbrace k \rbrace$-torsion for some $k\geqslant 1,$ in conformity with
Conjecture~\ref{conj:intFinConj}.

For any manifold~$M,$ the skein module $\Sk(M)$ has a natural
$H_1(M,\Z/2)$-grading, as the Kauffman relations always involve links in~$M$
that have the same homology class in $H_1(M,\Z/2).$
Thanks to the basis we computed, we can answer some other questions raised
in~\cite{GM18} about the structure of quantum invariants of links
in $\Sigma \times S^1:$
\begin{corollary}\label{cor:imgGMmap}
  For any $z \in \Sk(\Sigma \times S^1),$ the image of~$z$ by Gilmer-Masbaum's
  evaluation map is of the form
  \[ \mathrm{ev}(z)=\underset{i \in I}{\sum}R_i(A)p^{i} \]
  for some rational functions
  $R_i(A) \in \QQ(A),$ and where
  $I=\lbrace g-1,g+1,\ldots 3g-3 \rbrace \cup \lbrace g \rbrace.$
  Moreover, the Gilmer-Masbaum map is injective on each graded subspace
  of $\Sk(\Sigma \times S^1).$
\end{corollary}
The corollary results from the fact that Gilmer and Masbaum showed that it
is the case for elements that are linear combinations of non-separating
simple closed curves with $0$ or $1$ arrows, and the trivial curves with
$0$ to $2g$ arrows.
In particular, the rational functions $R_i(A),$ as linear combinations of the coefficients
in the basis $\mathcal{B},$ are algorithmically computable invariants
of links in $\Sigma \times S^1,$ that satisfy the Kauffman relations.

Let us stress that Theorem~\ref{thm:basis_skein}
does not imply Conjecture~\ref{conj:intFinConj} for $M=\Sigma\times S^1$;
one would need to prove in particular that the $R_i(A)$ have bounded
denominators.
This is related to the work of March\'e and Santharoubane~\cite{MS16},
who defined a Jones like polynomial invariant for links~$L$ in
$\Sigma \times S^1,$ considering the highest order of the asymptotics of
the quantum invariants $RT_p(\Sigma\times S^1,L).$
Their invariant, which has value in $\Z[A^{\pm 1}]$ is closely related to
the invariant $R_{3g-3}(A).$
As our method is algorithmic, we computed the coefficients
in the basis $\mathcal{B}$ of a few arrowed diagrams.
These coefficients have interesting integral properties; they seem to be
in $\Z[A^{\pm 1}]$ instead of $\QQ(A)$, which
corroborates Conjecture~\ref{conj:intFinConj}.

We may also hope to find more direct formulas
for these coefficients: this may produce an alternative proof of
the linear independance of $\mathcal{B}$ as it would then suffice to prove
invariance of these coefficients by Reidemeister moves.
We hope to explore further these coefficients, which may be thought of as
Jones-like polynomial invariants for links in $\Sigma\times S^1$, in a later work.
\bigskip

The article is organized as follows.
In Section~\ref{sec:arrowed_diagrams}, we introduce the notion of arrowed
diagrams and the elementary moves they satisfy. In Section~\ref{sec:relations},
we introduce several important relations that we will use throughout the paper.
In the remainder of Section~\ref{sec:reducing_degree}, we define a notion of
degree on the set of arrowed multicurves, and by expressing multicurves as
linear combinations of multicurves of smaller degree, we prove that
$\Sk(\Sigma \times S^1)$ is spanned by arrowed trivial curves and arrowed
non-separating curves. In Section \ref{sec:elimin_arrows}, we show that one
only needs up to $1$ arrow on non-separating curves, and up to~$2g$ arrows
on the trivial curve to span $\Sk(\Sigma \times S^1).$ Finally, in
Section~\ref{sec:equiv_classes}, we introduce an equivalence
relation
on the set of
non-separating curves that is motivated by relations in the skein module.
We compute the equivalence classes of this relation, and deduce that
non-separating curves (with same number of arrows) that represent the same
element in $H_1(\Sigma,\Z/2)$ also represent the same element of
$\Sk(\Sigma \times S^1),$ concluding the proof of Theorem~\ref{thm:basis_skein}.

{\em Acknowledgements.} The first author was supported by the
Max Planck Institute for Mathematics during this research, and thanks the
institute for its hospitality. The second author learned about this problem, and about
TQFT more generally, from Julien March\'e while he was writing his introductory
text~\cite{JulienIntro}.
We are also both grateful to Gregor Masbaum for
his constant interest in this work.

\section{Arrowed diagrams, complexity and trivial components}
\label{sec:arrowed_diagrams}
In this section, we introduce the notion of arrowed multicurves and arrowed
diagrams on~$\Sigma.$ While the definition of elements in the skein module
$\Sk (M)$ of a $3$-manifold~$M$ is very much $3$-dimensional, this notion
which will give us a purely $2$-dimensional way of thinking of elements of
$\Sk (\Sigma \times S^1).$
Arrowed diagrams were first
introduced by Dabkowski and Mroczkowski in~\cite{DM09}, where they were used
to compute the skein module of $\Sigma_{0,3}\times S^1.$

Let us view the $S^1$ factor of $\Sigma \times S^1$ as $[0,1]/_{0 \sim 1},$
and let $L$ be a link in $\Sigma \times S^1.$ By a general position argument,
up to isotopy $L$ can be assumed to be transverse to
$\Sigma \times \lbrace 0 \rbrace,$ to have no vertical tangent and, furthermore,
the image of $L$ by the projection $\Sigma \times S^1 \rightarrow \Sigma$ can
be assumed to only have a finite number of double points with transverse intersection.
The diagram of the projection, together with the choice of upper and lower strand at
each double point/crossing, is almost enough to determine $L$ up to isotopy.
The only missing information is when does $L$ cross the level
$\Sigma \times \lbrace 0 \rbrace.$ Thus we add an arrow on the projection at
each point of the projection coming from an intersection point
$L \cap \left(\Sigma \times \lbrace 0 \rbrace \right).$ Moreover, we choose
the direction of the arrow to be the direction in which $L$ crosses
$\Sigma \times \lbrace 0 \rbrace$ positively.
Conversely, any arrowed diagram gives rise to a link in $\Sigma \times S^1$ in an
obvious way. Because the Kauffman bracket skein module deals with framed links,
we would like to put a canonical framing on each arrowed diagram. We do so by
choosing the parallel of $L$ to be the push-off of $L$ along the positive
direction of $S^1.$ With this convention, any framed link $L$ still has an arrowed
diagram, as we can always correct the framing by adding curls to the components of
the diagram.

Dabkowski and Mroczkowski gave a complete set of moves describing isotopy of
framed links in $\Sigma \times S^1:$
\begin{proposition}\label{prop:Reidemeister}\cite{DM09}
  Two arrowed diagrams of framed links in $\Sigma \times S^1$ correspond to
  isotopic links if and only if they are related by standard Reidemeister
  moves $R_1',$ $R_2,$ $R_3$ and the moves:
  \begin{center}
  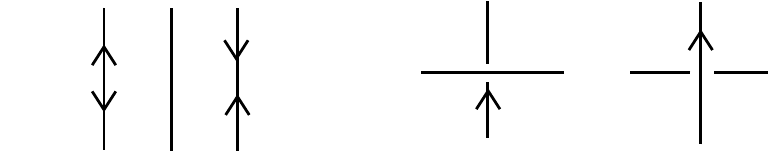
  \end{center}
\end{proposition}
The relations $R_4$ and $R_5$ imply the commutation relation:
if $\gamma \cdot \delta$ is the link obtained by stacking the diagram
$\gamma$ on top of $\delta,$ then $\gamma \cdot \delta=\delta \cdot \gamma.$
It is easy to see directly that those two links are isotopic in $\Sigma \times S^1.$

By relation $R_4,$ we note if a strand of an arrowed diagram has some number
of arrows (maybe with different directions) in succession, only the algebraic
number of arrows matters.
It will often be useful to write an arrow indexed by $n \in \Z$ to denote~$n$
successive arrows all pointing to the direction of the arrow.
If~$n$ is negative, it has to be understood as $|n|$ arrows pointing in
the opposite direction.

In all this article, by an {\em arrowed multicurve}, or
{\em multicurve} for short, we mean an arrowed diagram without double points.
In other words, this is a diagram whose projection on~$\Sigma$ is a
$1$-dimensional submanifold of~$\Sigma$. Each of its components may be
homotopically trivial, or essential, and we count these curves to form the
degree and complexity of the diagram.
\begin{definition}\label{def:degree}
  If $\gamma$ is an arrowed multicurve and $n$ (resp. $m$) are its number
  of essential, non-separating (resp. non-trivial separating) simple closed
  curve components, then we define
  \[ \deg (\gamma)=n+2m.\]
  We also define the complexity of $\gamma$ as
  $(\deg(\gamma), n+m).$
  Complexities are ordered using the lexicographic order.
\end{definition}
Notice that the above notion of degree does not depend on the number of
arrows that decorate each component of $\gamma.$

Note that in the definition of the degree and complexity, we do not count
the trivial components of $\gamma$.
This is because we can essentially get rid of them, as follows.
\begin{proposition}\label{prop:multicurves_span}
  For $\Sigma$ a closed compact oriented surface, the skein module
  $\Sk(\Sigma \times S^1)$ is spanned by arrowed multicurves containing no
  trivial component, and by the arrowed multicurves consisting of just the
  trivial curve with any number of arrows.
  
  Moreover, every arrowed multicurve is a linear combination of arrowed
  multicurves as above and with same degree and complexity.
\end{proposition}
\begin{proof}
  It is clear that repeatedly applying Kauffman relations $K_1$ to the
  crossings of an arrowed diagram will express it as a $\QQ(A)$-linear
  combination of arrowed multicurves. Thus the main point of
  Proposition~\ref{prop:multicurves_span} is its second assertion:
  we can eliminate a trivial component (with arrows) if the multicurve
  has at least one other component, without changing
  its degree or complexity.
  
  Let us consider an arrowed multicurve $\gamma$ containing a trivial curve
  and another component. If the trivial curve contains no arrow then the
  Kauffman relation $K_2$ gets rid of it. Otherwise, we use the relation:
  \begin{center}
  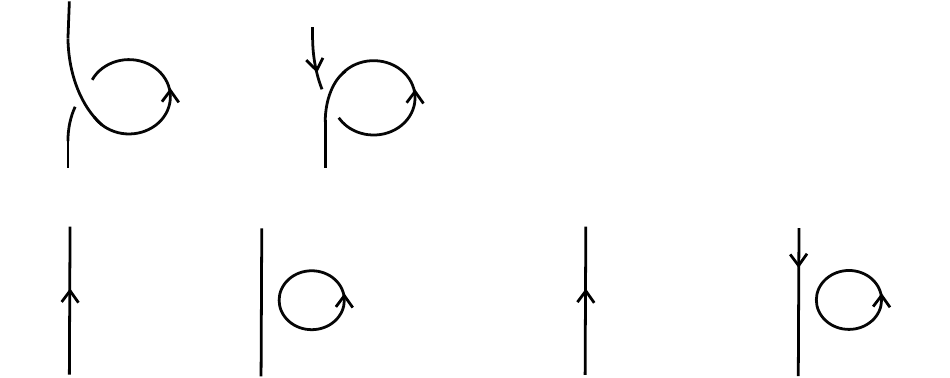
  \end{center}
  In the above, the left strand stands for a strand of another component of
  the starting multicurve (which is the second term of the left hand side
  of the second equality).
  
  By using this relation, we can express $\gamma$ as a
  linear combination of arrowed mulicurves where the number of arrows on
  a trivial component has increased, or decreased: this number of arrows can
  therefore be pushed to $0,$ and then the Kauffman relation $K_2$ gets rid
  of that trivial component.
  We then proceed inductively to erase the trivial components, until there
  is either no trivial component left, or just one trivial component and no
  other component.
\end{proof}

\section{Reducing the degree of multicurves}
\label{sec:reducing_degree}
The main result of this long section will be the following.
\begin{proposition}\label{prop:generatedbydegree2}
  The skein module $\Sk (\Sigma \times S^1)$ is spanned by arrowed multicurves
  of degree $\leqslant 1.$ That is $\Sk(\Sigma \times S^1)$ is spanned by
  arrowed non-separating curves and by arrowed trivial curves.
\end{proposition}

In the next subsection, we first introduce a few helpful
relations that relate a few multicurves of the same degree, up to lower degree
terms. We will use them to prove Proposition~\ref{prop:generatedbydegree2}
inductively, showing that any arrowed multicurve of degree $n\geqslant 2$
is a linear combination of arrowed multicurves with smaller degree.

\subsection{The sphere and the torus relation}
\label{sec:relations}
Our first relation relates multicurves that sit on a $5$-holed sphere
subsurface of~$\Sigma:$
\begin{proposition}\label{prop:sphere_relation}
  For any $n\geqslant 1,$ we have the sphere relation $(S_n)$ between
  multicurves of degree $2n+6:$
  \begin{center}
  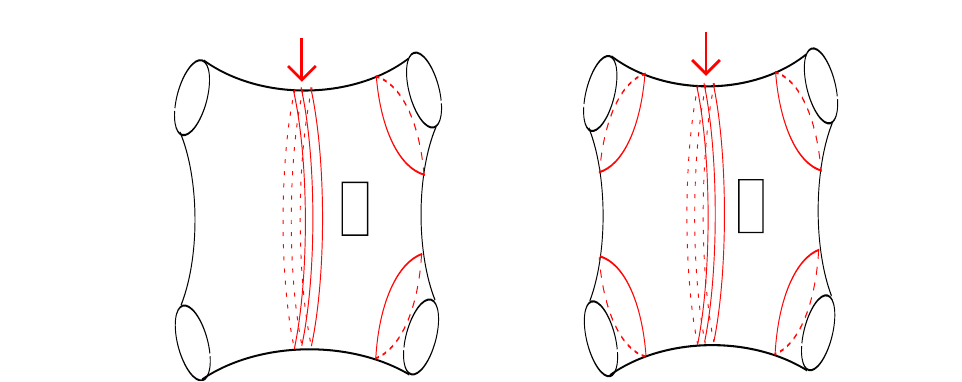
  \end{center}
  In the above figure, the two boundary components on the left are non-separating
  curves of the ambiant surface~$\Sigma,$, while all the other red curves
  are essential, separating curves of~$\Sigma.$
  The black square may be homotopically trivial in~$\Sigma$, or not.
  Finally, $\equiv$ is equality up to a linear combination of multicurves
  with degrees $\leqslant 2n+5.$
\end{proposition}
We note that the multicurves above are indeed of degree $2n+6:$ as the middle~$n$
components and the two rightmost ones are separating they contribute $2n+4$ to
the degree, and for each side one adds either a separating curve of two
non-separating curves, adding a total of $2$ more to the degree.
We will sometimes refer to the relations $(S_n)$ as the \textit{sphere relations}.
\begin{proof}
  \begin{figure}[!h]
  \centering
  \def \svgwidth{1.1\columnwidth}
  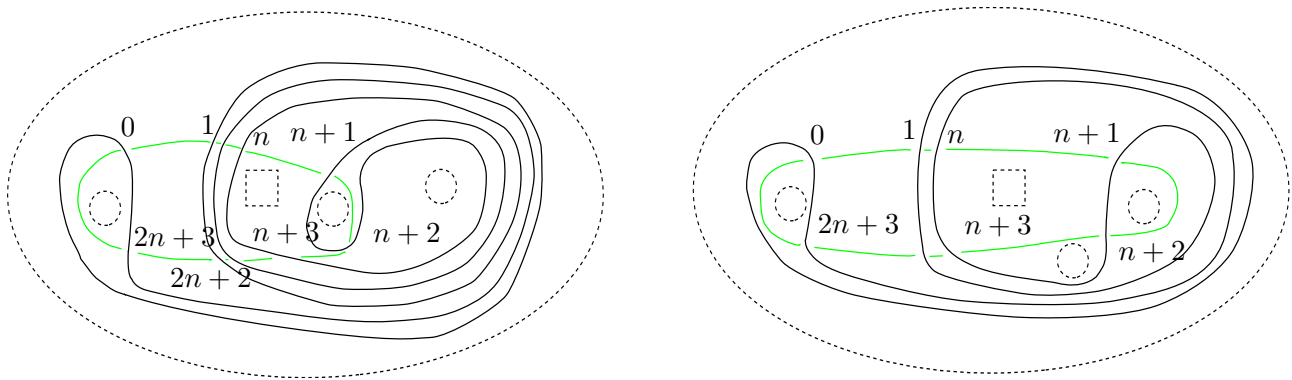
  \caption{The black curve $\gamma$ stands atop the green curve $\delta$ inside a $4$-holed disk.
  The two rightmost boundary components are the rightmost boundary components in
  Proposition~\ref{prop:sphere_relation}. There are $n\geqslant 1$ strands of
  $\gamma$ going in between the two squares, and the left (resp. right) diagram
  correspond to whether $n$ is odd or even. In each diagram, we labeled the
  crossings from $0$ to $2n+3,$ following the green curve.}
  \label{fig:sphere_resolution}
  \end{figure}
  We prove Proposition~\ref{prop:sphere_relation} by considering the resolutions
  of the diagram shown in Figure~\ref{fig:sphere_resolution}. In this diagram,
  the $5$-holed sphere is drawn as a four-holed disk, with the boundary of the
  disk and the leftmost hole corresponding to the two boundary components on the
  left of Relation~$(S_n).$ We have to consider two slightly different patterns
  depending whether~$n$ is odd or even.
  The diagram shows two simple closed curves~$\gamma$ and~$\delta,$ with~$\gamma$
  (in black) standing on top of~$\delta$ (in green). We will produce Relation~$(S_n)$
  from the commutation relation $\gamma \cdot \delta= \delta \cdot \gamma,$ after a
  careful study of the multicurves
  that appear after resolving the crossings using Kauffman relations.
  We recall that in $\gamma \cdot \delta$ and $\delta \cdot \gamma$ the exact same
  resolutions appear, with coefficients changed by replacing $A$ by $A^{-1}.$
  With this in head, it is sufficient to study the resolutions of $\gamma \cdot \delta.$
  There are $2n+4$ crossings between $\gamma$ and $\delta,$ and thus $2^{2n+4}$
  resolutions.
  Our claim is that among all those resolutions, there are only $2$ of maximal
  degree, and that this maximal degree is $2n+6.$
  
  Notice that the $2n+4$ crossings
  cut the black and green curves~$\gamma$ and~$\delta$ into $2n+4$ black arcs
  and $2n+4$ green arcs. Moreover, any component of any resolution consists of
  several green and black arcs, alternating between green and
  black:
  any such component is composed of an even number of arcs.
  We will focus mainly on the green arcs, and denote them by $(i,i+1)$
  with $0\leqslant i\leqslant 2n+3$, with cyclic notation,
  as suggested in Figure~\ref{fig:sphere_resolution}.
  
  Given a resolution $\lambda$ of $\gamma\cdot\delta$ and a green arc~$a$, consider the
  connected component $\lambda'$ of~$\lambda$ containing~$a$. We will say that the
  {\em contribution of~$a$ to the degree of~$\lambda$} equals~$\frac{p(a)}{q(a)}$, where
  $q(a)$ is the number of green arcs in $\lambda'$, and where
  $p(a)$ equals $0$ if~$\lambda'$ is a non essential curve of $\Sigma$,
  $1$ if it is a non-separating simple closed curve, and $2$ if it is
  a non-trivial separating curve. This way, $\deg(\lambda)$ is the sum of
  the contributions of its green arcs.
  
  Let us bound, individually, the contribution each green arc can
  have to the degree of a resolution of $\gamma\cdot\delta$.
  The arcs $(2n+3,0)$, $(0,1)$, $(n+1,n+2)$ and $(n+2,n+3)$ are the only ones
  that can form a closed curve containing no other green arcs: for the two
  first ones that curve is non-separating, for the two others it is separating.
  Hence the maximal contributions they can give to the degree are respectively
  $1$, $1$, $2$ and $2$. For any other green arc $a$, we will have
  $p(a)\leqslant 2$ and $q(a)\geqslant 2$, hence the contribution $a$ cannot exceed~$1$.
  By summing up all these contributions, we deduce that for all resolution~$\lambda$
  of $\gamma\cdot\delta$ we have $\deg(\lambda)\leqslant 2n+6$.
  
  Finally, let us examine in which resolutions of $\gamma\cdot\delta$ the
  degree $2n+6$ can indeed be reached. The contributions of both green
  arcs $(n+1,n+2)$ and $(n+2,n+3)$ need to equal~$2$. For this, these green
  arcs have to be matched to one black arc to make a separating curve in~$\Sigma$.
  Thus, the crossings $n+1$, $n+2$ and $n+3$ need to receive the
  resolutions~$-$,~$+$ and~$-$ respectively. Now in order to reach maximal contribution
  to the degree, the green arc $(n, n+1)$ has to be paired
  with another green arc, and two black arcs, to form a separating curve
  of~$\Sigma$. The only possibility for that is to be paired with the green
  arc $(n+3,n+4)$, and the crossings numbered $n$ and $n+3$ have to receive
  both the resolution~$-$. The same argument with the arc $(n-1,n)$ implies
  it is paired with the arc $(n+4,n+5)$ and the crossings $n-1$ and $n+4$
  have the resolution~$-$. We continue further left, until the arc $(1,2)$
  which is paired with the arc $(2n+2,2n+3)$ (it can also be paired with
  the arc $(2n+3,0)$ to form a closed curve, but that curve is trivial in~$\Sigma$).
  In conclusion, all crossings except $n+2$ and maybe~$0$, have to receive
  the resolution~$-$.
  
  Depending on the resolution of the crossing labeled~$0$, we have two possible
  diagrams of maximal degree. The one for which the resolution of~$0$
  is~$-$ yields first diagram of the relation~$(S_n)$. It has $2g-3$ signs~$-$
  and one sign~$+$ hence it comes with coefficient $A^{-2g+2}$ in $\gamma\cdot\delta$,
  and with coefficient $A^{2g-2}$ in $\delta\cdot\gamma$. The positive resolution
  of~$0$ yields the second diagram, with coefficient $A^{-2g}$ in $\gamma\cdot\delta$
  and $A^{2g}$ in $\delta\cdot\gamma$. The equality $\gamma\cdot\delta=\delta\cdot\gamma$
  in the skein module thus proves the formula~$(S_n)$.
\end{proof}
We defined the sphere relation $(S_n)$ for any $n\geqslant 1.$ We will need another
relation which we will call $(S_0).$ It is again a relation between multicurves on
a $5$-holed sphere subsurface, although it has a slightly different form.
\begin{proposition}\label{prop:sphere_relation2}
  In $\Sk( \Sigma \times S^1),$ we have the following relation between multicurves
  that coincide except in a $5$-holed sphere:
  \begin{center}
  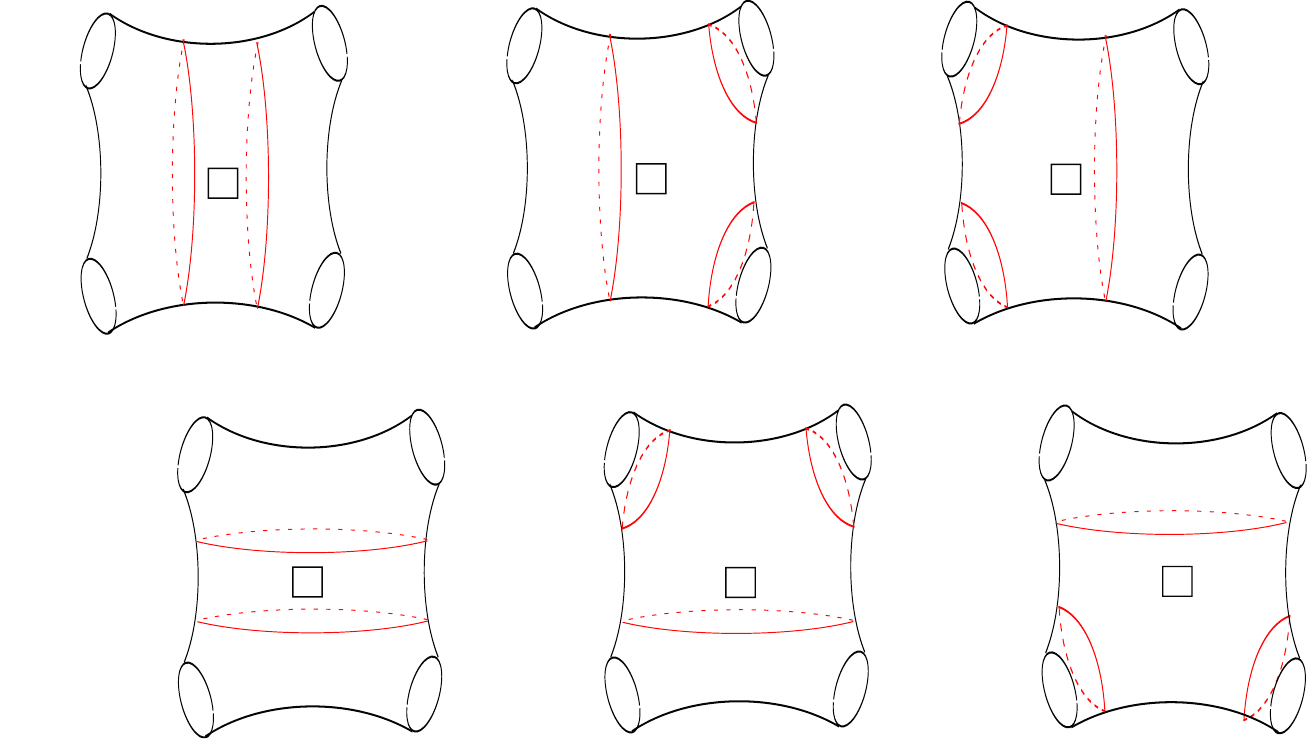
  \end{center}
\end{proposition}
\begin{proof}
  We get the relation looking at resolutions of $\gamma \cdot \delta,$ where
  $\gamma$ and $\delta$ are the following simple closed curves:
  \begin{center}
  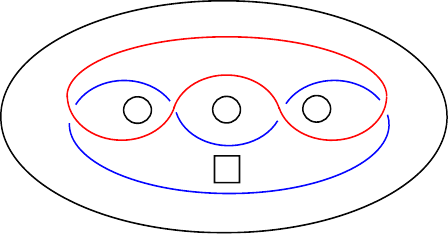
  \end{center}
  In the above, we numbered some boundary components from $1$ to $4,$ where $1$
  and $2$ correspond to the two leftmost boundary components in
  Proposition~\ref{prop:sphere_relation2} and $3$ and $4$ to the two rightmost
  components. We get the sphere relation $(S_0)$ from the equality
  $\gamma \cdot \delta=\delta \cdot \gamma.$ Note that resolutions with an
  even number of positive resolutions at crossings will appear with the same
  coefficient $1$ on both sides, so we need to analyse the other resolutions only.
  Let us order the crossings from left to right, the following sums up the
  different resolutions:
  \begin{center}
  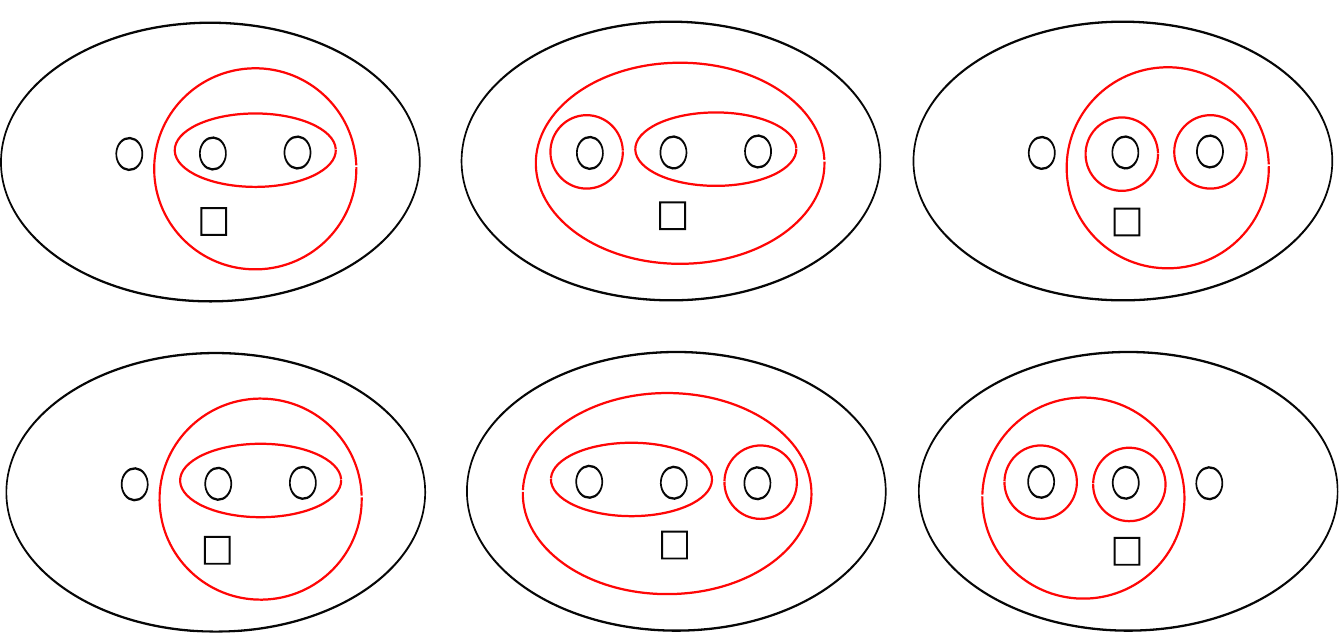
  \end{center}
  with the remaining odd resolutions $+-++,$ $+++-,$ $--+-$ and $+---$ all
  yielding trivial curves. Thus collecting all terms of the equality
  $\gamma \cdot \delta=\delta\cdot \gamma$ we indeed get the relation~$(S_0).$
\end{proof} 
We now establish another useful relation which we will call the Torus relation.
This is a relation between multicurves that sit on a torus (with several holes)
subsurface of $\Sigma.$ We have:
\begin{proposition}\label{prop:torus_relation}
  Let $\gamma$ be an arrowed multicurve. We suppose that
  there exists a simple curve $\delta$ which intersects exactly once $n\geqslant 3$
  components of $\gamma$ and which is disjoint from all other curves of $\gamma$.
  Then, up to multicurves of lower degree, $\gamma$ is a linear combination of
  arrowed multicurves obtained from $\gamma$ by replacing any number of pairs
  of consecutive components of $\gamma$ along $\delta$ by their connected sum
  along the arc of $\delta$ connecting them, provided the resulting curves
  are non-trivial and separating.
\end{proposition}
Note that in the torus relation all multicurves are of degree~$n:$ all components of
$\gamma$ are non-separating as they intersect~$\delta$ once, and the other
multicurves have the same degree as we always replace two non-separating curves
with one separating one.
We also remark that any multicurve that satisfy the hypothesis of
Proposition~\ref{prop:torus_relation} is a linear combination of multicurves
with smaller complexity. If moreover no pair of consecutive components bounds
a subsurface $\Sigma'\subset \Sigma$ with genus $\geqslant 1,$ then~$\gamma$ is
actually a linear combination of multicurves of smaller degree, as those connected
sums are all either trivial curves or non-separating curves.
\begin{proof}
  \begin{figure}[!h]
  \centering
  \def \svgwidth{0.2\columnwidth}
  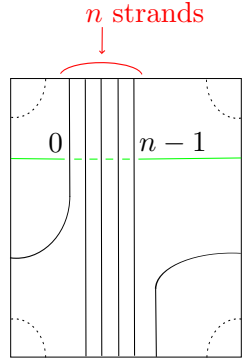
  \caption{The black curve $\gamma'$ stands atop the green curve $\delta.$ The dashed
    quarter-circles correspond to the boundary component of the torus. We labeled the
    crossings from $0$ to $n-1,$ from left to right. The surface is represented as a
    torus using the square model, but regions of the diagram may actually contain some
    genus, or be connected by handles}
  \label{fig:torus_resolution}
  \end{figure}
  We prove Proposition~\ref{prop:torus_relation} using a similar method as for
  Proposition~\ref{prop:sphere_relation}.
  Let us introduce the curve~$\gamma'$
  which is the $1/n$ fractional Dehn twist of~$\gamma$ along~$\delta.$
  The curves~$\gamma'$ and~$\delta$ are represented on Figure~\ref{fig:torus_resolution}.
  We will deduce the relation from the equality $\gamma'\cdot \delta=\delta\cdot \gamma'.$
  This time, the diagram shown in Figure \ref{fig:torus_resolution} has $n$ crossings,
  thus $n$ green arcs and $n$ black arcs. The crossings are labeled $0$ to $n-1,$ and
  considered as elements of $\Z_n.$
  Notice that the black arcs connect crossings $i$ to crossing $i+1.$
  We will find the bigons and squares of the diagram, and show that no bigon corresponds
  to a separating curve, which will imply that the maximal degree is at most $n.$
  As there are both green and black arcs connecting $i$ to $i+1,$ we get exactly $n$
  bigons, and each green arc belongs to exactly $1$ bigon. Notice that here it is
  important that $n\geqslant 3:$ otherwise there are actually two black arcs that close up
  the green arc $(0,1),$ as $1+1=0 \ \mathrm{mod} \ 2.$
  Each of those bigons corresponds to a vertical curve in the square representing the
  one-holed torus: that is, the bigons correspond to parallel non-separating curves.
  
  Let us now search for the squares $(i,i+1,j,k),$ where $(i,i+1)$ is a green arc.
  Given the form of black arcs, they must be of the form $(i,i+1,i+2,i+1).$
  We note that those squares all correspond the curves that are connected sums of
  consecutive components of~$\gamma$ along an arc of~$\delta.$
  In a maximal degree resolution, all green arcs must belong to either a bigon
  (which has to be a non-separating curve as we saw) or a square that is a
  separating curve. So the different resolutions are exactly~$\gamma$ and the
  different possible multicurves described in the first part of
  Proposition~\ref{prop:torus_relation}.
  Thus the equality $\gamma' \cdot \delta=\delta \cdot \gamma'$ gives exactly
  the equation of Proposition~\ref{prop:torus_relation}.
\end{proof}
Next we introduce a relation between arrowed multicurves in a two-holed torus
subsurface of $\Sigma.$ Those relations will involve mutlicurves with at most~$2$
components (in the subsurface), which are not covered by Proposition~\ref{prop:torus_relation}.
\begin{proposition}\label{prop:2holed_torus_relation}
  We have the relations:
  \begin{center}
  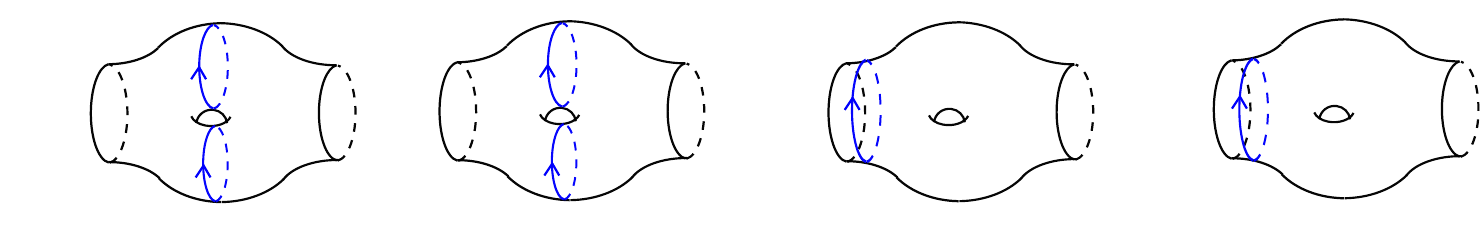
  \end{center}
  and
  \begin{center}
  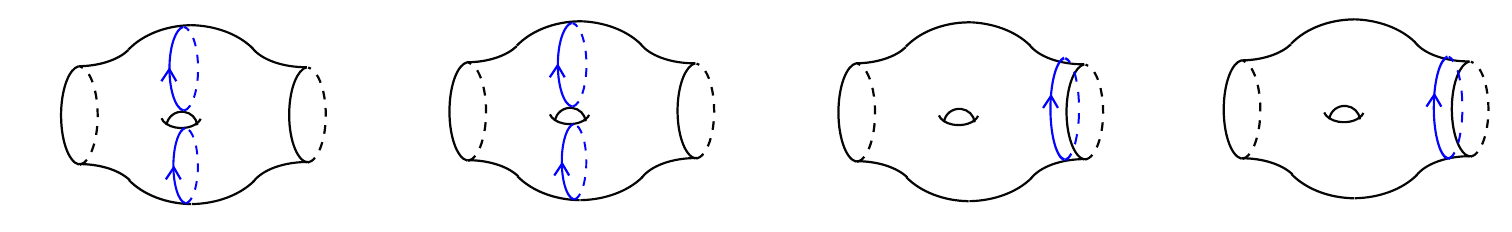
  \end{center}
  In particular, the multicurves on the left hand side of two equations are linear
  combinations of multicurves of smaller complexity.
\end{proposition}
\begin{proof}
  From Reidemeister moves $R_5$, we have:
  \begin{center}
  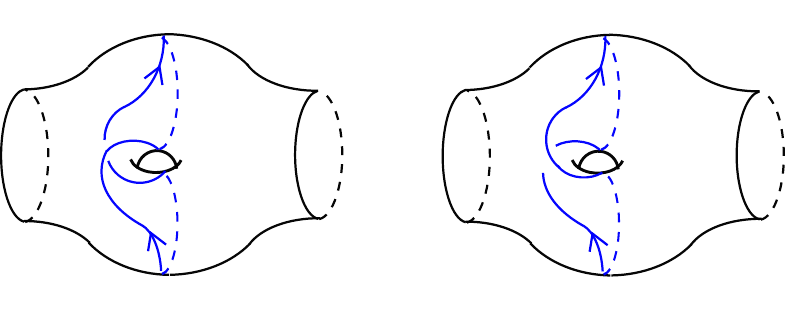
  \end{center}
  After resolving the crossings, this gives the first equation. Similarly, the
  second equation is a consequence of the relation:
  \begin{center}
  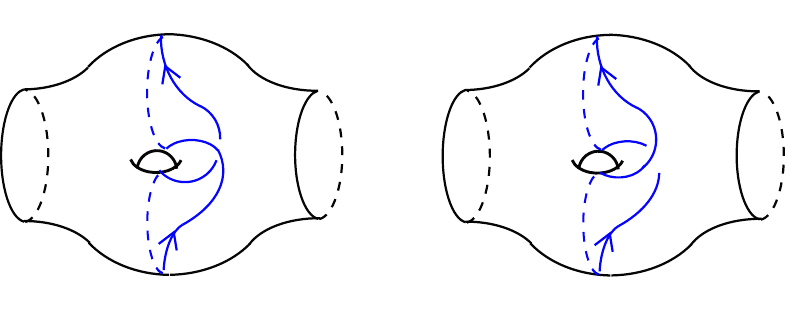
  \end{center}
  Finally, let us call $D_{a,b},$ $D_{a+1,b+1}$ the two multicurves on the left
  hand of the two equations, let $\equiv$ be equality modulo multicurves of
  smaller complexity. Note that $D_{a,b}$ is composed of two non-separating
  curves, so it has degree $2.$ The multicurves on the right-hand side have
  only one component, thus have degree at most $2$ and smaller complexity
  than $D_{a,b}.$
  Thus we have $A^{-1}D_{a,b}-A D_{a+1,b+1}\equiv 0$ and
  $A D_{a,b}-A^{-1}D_{a+1,b+1}\equiv 0,$ which implies that
  $\lbrace 2 \rbrace D_{a,b}\equiv 0$ and $\lbrace 2 \rbrace D_{a+1,b+1}\equiv 0.$
  So the multicurves on the left hand side are linear combinations of multicurves
  of smaller complexity. 
\end{proof}

\subsection{The dual graph of an arrowed multicurve}
\label{sec:dual_graph}
Given an (arrowed) multicurve in $\Sigma$ we define its dual tree as follows:
\begin{definition}
  Let $\gamma$ be a multicurve in $\Sigma.$ Let $c$ be the multicurve that
  consists of copy of each distinct homotopy classes of non-trivial separating
  curves among components of~$\gamma,$ and let~$V$ be the set of connected
  components of $\Sigma \setminus c.$ Then the graph~$\Gamma$ dual to~$\gamma$
  has one vertex for each element of~$V$ and one edge for each component of~$c,$
  connecting the two connected components of $\Sigma \setminus c$ that it bounds.
  For each vertex $v\in V$, we associate the corresponding
  connected component $\Sigma(v)\subset\Sigma$. We let $g(v)$ denote its genus,
  and by $\gamma\cap\Sigma(v)$ we mean the arrowed diagram of $\Sigma(v)$
  consisting of the non-separating connected components of $\gamma$ lying
  inside $\Sigma(v)$.
\end{definition}
Note that the graph $\Gamma$ is actually a tree as any edge of $\Gamma$ is disconnecting.

For $\gamma \in \Sk(\Sigma \times S^1)$ a multicurve, we say that~$\gamma$ is
\textit{stable} if it is not a linear combination of multicurves of smaller complexity.
\begin{lemma}\label{lemma:vertex}
  Let $\gamma$ be a stable multicurve and let $v$ be a vertex of the dual graph of~$\gamma.$
  Then $\gamma \cap \Sigma(v)$ consists of either $0$ or $1$ non-separating curve.
\end{lemma}
\begin{proof}
  Let $\gamma'$ be the multicurve $\gamma \cap \Sigma(v),$ which consists only
  of non-separating curves. By the Torus relation, if we can find a curve~$\delta$
  which intersects at least $3$ components of $\gamma'$ exactly once, then~$\gamma$
  is a linear combination of multicurves of smaller complexity, hence $\gamma$ is not stable.
  Also, if we can find a curve~$\delta$ which intersects
  exactly two components of~$\gamma'$ once, then a neighborhood of the union
  of these three curves is a two-holed torus. Proposition~\ref{prop:2holed_torus_relation}
  then asserts that~$\gamma$ is not stable.
  
  Thus, we just need to prove that such a curve~$\delta$
  exists provided~$\gamma'$ contains at least two non-separating curves.
  For this, consider the dual graph $G(v)$ to~$\gamma'$ in the
  usual sense: its vertices are the connected components of
  $\Sigma(v)\smallsetminus\gamma'$ and each component of~$\gamma'$ yields one
  edge of $G(v)$. If $G(v)$ has at least two vertices then we can find
  an embedded loop in $G(v)$, which can be followed to define a curve~$\delta$
  as above. Similarly, if $G(v)$ has only one vertex, but at least two loops,
  then the two corresponding components of~$\gamma'$ are non-separating, and
  mutually non-separating in $\Sigma(v)$, hence we can find a curve~$\delta$
  intersecting just these two curves once: provided $G(v)$ has at least two edges, such a
  curve~$\delta$ exists.
\end{proof}
\begin{proposition}\label{prop:dual_graph}
  Let~$\gamma$ be a stable multicurve on~$\Sigma.$ Then the dual graph
  of~$\gamma$ is linear.
\end{proposition}
\begin{proof}
  Let $\gamma$ be a stable multicurve and~$\Gamma$ be its dual graph.
  Let us note that by definition of dual graphs, if a vertex~$v$ of~$\Gamma$
  has valency $\leqslant 2$ then its genus $g(v)$ is at least~$1.$
  Hence, it is then sufficient to prove that if two
  vertices~$v$ and~$v'$ are connected by an edge and if $g(v')\geqslant 1,$ then
  the valency of $v$ is at most $2.$ Indeed, assuming this claim, starting at
  the leaves of the tree~$\Gamma$ which have genus $\geqslant 1,$
  then neighboring vertices also have valency $\leqslant 2$ and thus genus $\geqslant 1.$
  Once can proceed inductively to prove that~$\Gamma$ is linear.
  
  Thus let us assume that $\Gamma$ has two connected vertices $v$ and $v'$ such
  that $g(v)\geqslant 1$ and suppose the valency of $v'$ is at least $2.$
  A neighborhood of the connected components of~$\gamma$ corresponding to the edge
  $(v,v')$ looks like the second term of the sphere relation of
  Proposition~\ref{prop:sphere_relation}, with $n\geqslant 1$ curves in the middle.
  
  If $n\geqslant 2,$ we have apply the sphere relation $(S_{n-1}),$ reducing the
  value of~$n$ while adding copies of the two leftmost boundary components.
  Doing so we actually increase complexity, but we keep the same degree.
  We inductively reduce the value of~$n$ until we hit~$n=1.$
  Finally we apply the relation $(S_0)$ of Proposition~\ref{prop:sphere_relation2},
  which now expresses the multicurve we got in terms of multicurves of smaller
  degree. We deduce that the multicurve~$\gamma$ was not stable.
\end{proof}
We end this section with a lemma that refines Lemma~\ref{lemma:vertex} for
vertices of valency~$2:$
\begin{lemma}\label{lemma:valency2}
  Let $\gamma$ be a stable multicurve and~$v$ be a vertex of its dual graph~$\Gamma$
  of valency~$2.$ Then $\gamma \cap \Sigma(v)=\emptyset.$
\end{lemma}
\begin{proof}
  Because of Lemma~\ref{lemma:vertex}, we know that $\gamma \cap \Sigma(v)$ is
  either empty or a single non-separating curve.
  The subsurface $\Sigma(v)$ has genus $g\geqslant 1$ and $2$ boundary components.
  We will provide a relation to show that a multicurve in $\Sigma(v)$ consisting
  of the two boundary components and one non-separating curve is a linear combination
  of multicurves of smaller degree.
  
  Consider two curves $\gamma$ and $\delta$ in a surface $\Sigma_{g,2}$ as follows:
  \begin{center}
  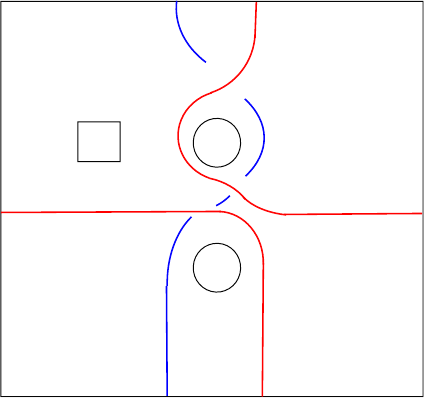
  \end{center}
  In the above, the two circles are the two boundary components of $\Sigma_{g,2},$ and
  the opposite sides of the big square are identified. Moreover, we attach $g-1$
  handles to the little square, so that the surface indeed has genus~$g.$
  
  Let us order the crossings from bottom to top.
  We claim that the resolution~$++-$
  consists of the two boundary components plus one non-separating curve, which gives
  total degree~$5,$ and that all other resolutions have degree at most~$3.$
  The different resolutions are summed up in the following diagrams:
  \begin{center}
  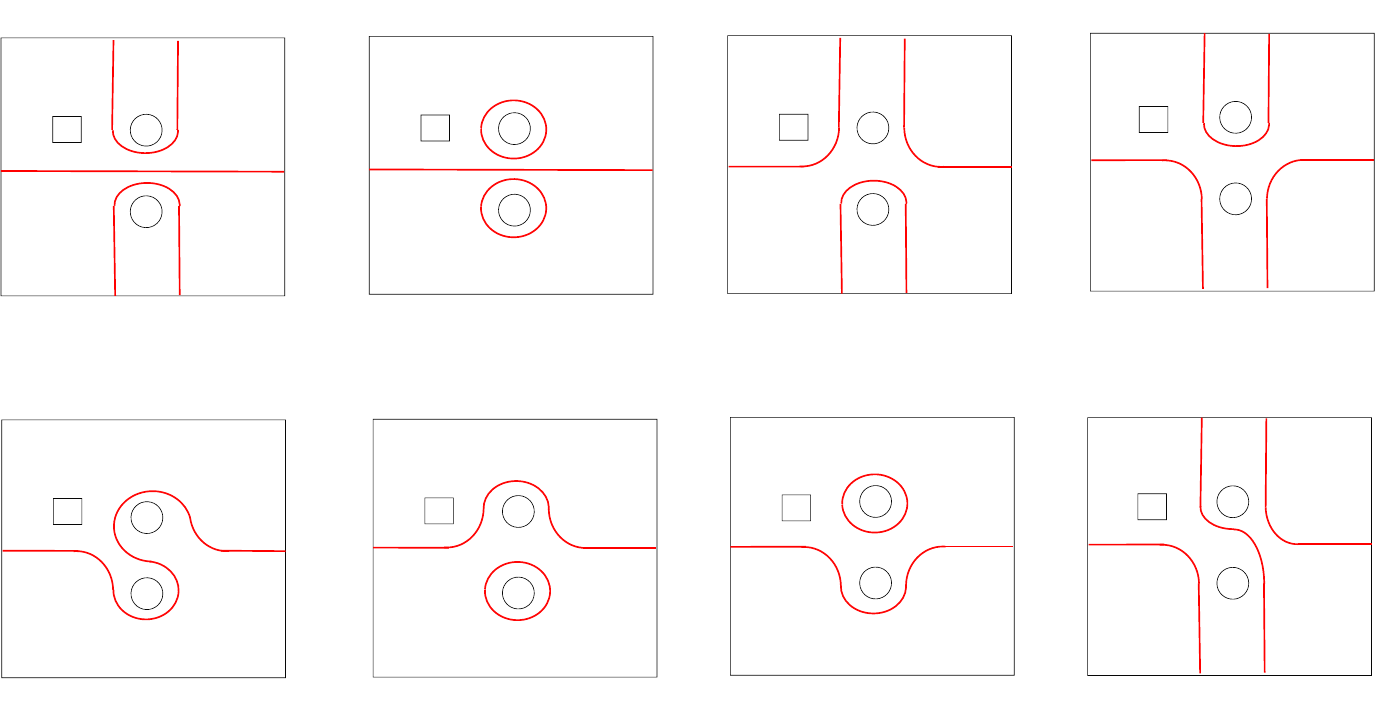
  \end{center}
  As the resolution~$++-$ appears in $\gamma \cdot \delta-\delta\cdot \gamma$ with
  coefficient $\lbrace 1 \rbrace\neq 0,$ that multicurve is a linear combination
  of multicurves with smaller degree, which is what we wanted.
\end{proof}

\subsection{Sausage decompositions of surfaces}
\label{sec:sausage}
Thanks to Proposition~\ref{prop:dual_graph} and Lemma~\ref{lemma:vertex}, the
skein module of $\Sigma \times S^1$ is generated by arrowed multicurves which
are put on~$\Sigma$ in a kind of standard form, that fits well with a special
kind of pair of pants decomposition of~$\Sigma$ which we will call a
\textit{sausage-decomposition} of $\Sigma.$ We define such a decomposition below:
\begin{definition}
  Let $\Sigma$ be an oriented compact closed surface of genus $g.$
  A sausage-decomposed subsurface $\Sigma'\subset \Sigma$ is the data of a
  subsurface of~$\Sigma$ with~$2$ to~$4$ boundary components together with a
  pair of pants decomposition of the type described in Figure~\ref{fig:sausage}
  with pair of pants being ordered from left to right.
  Moreover, a sausage decomposition of~$\Sigma$ is the data of a sausage subsurface
  of~$\Sigma$ composed of $2g-2$ pair of pants, and with $2$ boundary components
  that each bound a disk in~$\Sigma.$
\end{definition}
\begin{figure}[hb]
  \centering
  \def \svgwidth{0.6\columnwidth}
  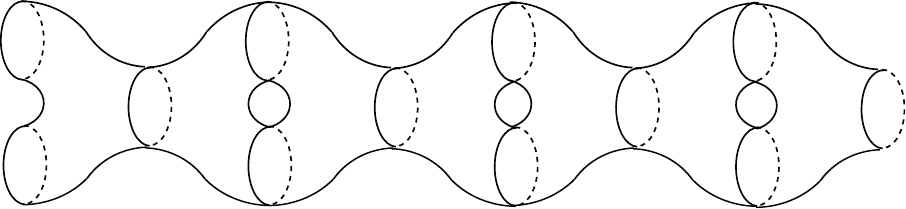
  \caption{A sausage-decomposed subsurface of $\Sigma$}
  \label{fig:sausage}
\end{figure}
Let us remark that in a sausage decomposition of~$\Sigma$ there is a well-defined
left (and right) boundary component.

Let us a fix a sausage-decomposed subsurface $\Sigma' \subset \Sigma,$ containing
$N=|\chi(\Sigma')|$ pairs of pants. We fix an integer $m\geqslant 0$ and
$k_0\in \lbrace 1,\ldots N-1 \rbrace$ so that the subsurface which is the union
of the first $k_0$ pairs of pants of $\Sigma'$ has a single boundary component to its right.
Let $a,b \in \Z$ and $k \in \lbrace 0,\ldots ,N \rbrace\setminus \lbrace k_0 \rbrace,$
we write $D_{a,b}^k$ for the diagram (depending on the parity of $k$):
\begin{center}
  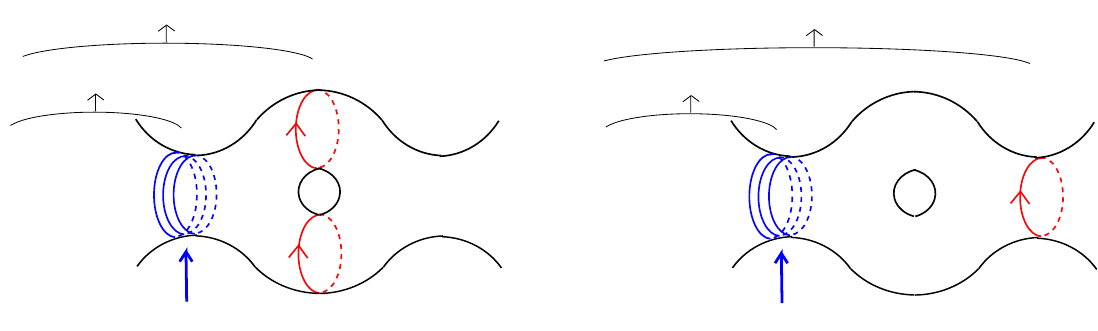
\end{center}
When $k=k_0,$ we would like to define a $D_{a,b}^k$ similarly, except we need
to specify the relative position of the blue and red curves. Thus we get two
versions $_lD_{a,b}^k$ and $_rD_{a,b}^k,$ where the red curve is put respectively
to the left and to the right of the $m$ blue curves.
\begin{center}
  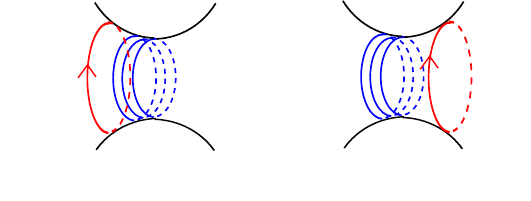
\end{center}
All those diagrams define elements of $\Sk(\Sigma \times S^1),$ which depend also
on $m$ and the sausage-decomposed subsurface $\Sigma',$ but for simplicity we
omit those dependence from the notations.

It is obvious that if $m=0$ then $_lD_{a,b}^{k_0}=_rD_{a,b}^{k_0}.$
There is a more general relation between those two diagrams which we describe
in the following lemma:
\begin{lemma}\label{lemma:leftright}
  For any $a,b \in \Z$ and $m \geqslant 0$ we have 
  \[ _lD_{a,b}^{k_0}\equiv A^{2m(a+b)} \ _rD_{a,b}^{k_0}, \]
  modulo diagrams of smaller degree.
\end{lemma}
\begin{proof}
  For any $a,b \in \Z$ we have:
  \begin{center}
  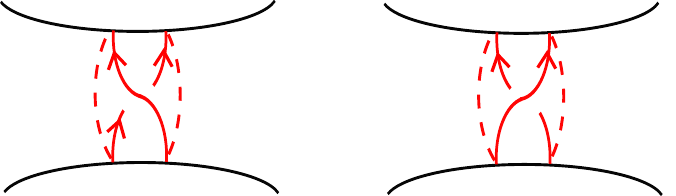
  \end{center}
  Thus:
  \begin{center}
  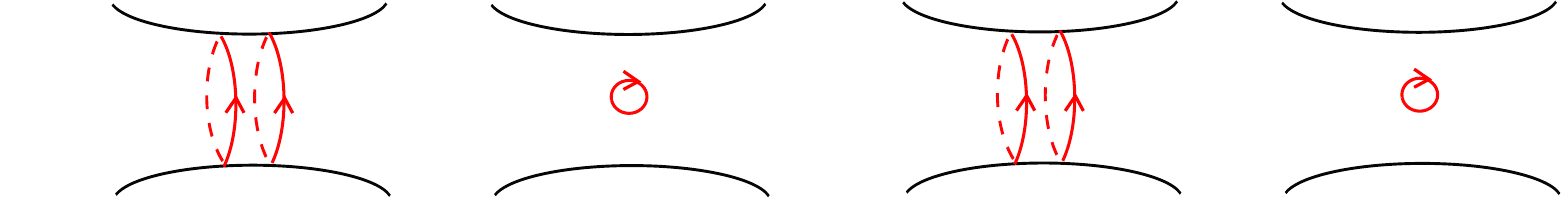
  \end{center}
  so that:
  \begin{center}
  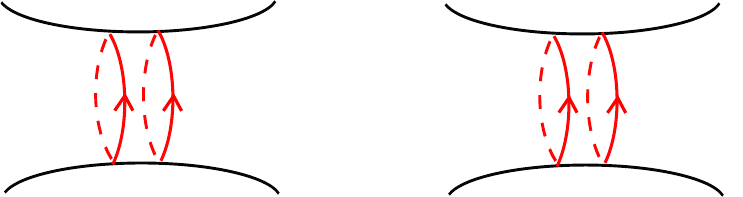
  \end{center}
  Thus we can push any arrow on any curve to the curve immediately to its
  right, at the expense of multiplying by $A^2$ each time.
  To push all arrows from the leftmost to the rightmost of the $m+1$ curves,
  we multiply by $A^{2m(a+b)}.$
\end{proof}
The next proposition says that we can push multicurves $D_{a,b}^k$ outward to
the boundary of the subsurface~$\Sigma'.$
\begin{proposition}\label{prop:pushout}
  Let $V^{\partial \Sigma'}$ be the subspace of $\Sk (\Sigma \times S^1)$
  spanned by the elements $D_{a,b}^0$ and $D_{a,b}^N.$ Then, up to diagrams
  of smaller degree, for any $a,b \in \Z,$ we have
  $_lD_{a,b}^{k_0}, _rD_{a,b}^{k_0} \in V^{\partial \Sigma'}.$
\end{proposition}
This proposition rests upon the following relations between the $D_{a,b}^k:$
\begin{lemma}\label{lemma:D_ab^k_relations}
  Let $k \in \lbrace 0,\ldots , N-1 \rbrace$ and $a,b \in \Z.$
  \begin{itemize}
  \item[-]If $k \notin \lbrace k_0-1,k_0\rbrace$ then
  \[ A D_{a,b}^k-A^{-1}D_{a+1,b+1}^k=A D_{a+1,b+1}^{k+1}-A^{-1} D_{a,b}^{k+1}. \]
  \item[-]If $k=k_0$ then
    \[ A \ _rD_{a,b}^k-A^{-1}\ _rD_{a+1,b+1}^k=A D_{a+1,b+1}^{k+1}-A^{-1} D_{a,b}^{k+1}. \]
  \item[-]If $k=k_0-1$ then
    \[ A D_{a,b}^k-A^{-1}D_{a+1,b+1}^k=A \ _lD_{a+1,b+1}^{k+1}-A^{-1} \ _lD_{a,b}^{k+1}. \]
  \end{itemize}
\end{lemma}
\begin{proof}
  Those equations are direct applications of the first or second two-holed
  torus relation of Proposition~\ref{prop:2holed_torus_relation}. We apply them
  in the $k+1$-th pair of pants of the decomposition of~$\Sigma',$ which may
  oriented to the right or to the left depending on the parity of~$k.$
  The cases where $k=k_0-1$ or $k_0$ work the same as the others, as we can
  always keep the extra curves away.
\end{proof}
Let $V$ be the $\QQ(A)$ vector space formally spanned by elements $D_{a,b}^k$
(and elements $_l D_{a,b}^{k_0}, \ _rD_{a,b}^{k_0}$) for $a,b \in \Z$ and
$k \in \lbrace 0, \ldots, N \rbrace.$ We will by a slight abuse of notation,
sometimes consider elements of~$V$ as elements of $\Sk (\Sigma \times S^1)$
that might be thus subject to relations.
We define on~$V$ a shift operator $s\colon V \rightarrow V$ by
$s(D_{a,b}^k)=D_{a+1,b+1}^k.$
It should be kept in mind that this operator $s$ acts
on {\em the space of diagrams $V$, before quotienting by the skein relations:}
it does not act on elements of the skein module.
Let also $A\colon V \rightarrow V$ be the multiplication operator by $A$
and consider
\[ \Delta_+=As-A^{-1}, \ \Delta_-=-A^{-1}s+A, \ \textrm{and}
\ \Delta_{+,m}=A^{2m+1}s-A^{-1}. \]
We note that $\Delta_{+,0}=\Delta_+.$
\begin{lemma}\label{lemma:D_ab^k_relations2}
  For any $a,b\in \Z$ we have $\Delta_+^{k_0} \ _lD_{a,b}^{k_0}=0$ and
  $\Delta_-^{N-k_0} \ _rD_{a,b}^{N-k_0} =0.$
\end{lemma}
\begin{proof}
  By Lemma~\ref{lemma:D_ab^k_relations}, we have
  $\Delta_+ D_{a,b}^k= \Delta_- D_{a,b}^{k+1},$ provided we do not run into
  the extra~$m$ curves.
  As the operators $\Delta_+$ and $\Delta_-$ commute, we get that
  $\Delta_+^2 D_{a,b}^k= \Delta_- D_{a,b}^{k+2}$ and so on, as long as we
  do not collide with the $m$ extra curves. In the end we get
  $\Delta_+^{k_0} \ _lD_{a,b}^{k_0}= \Delta_-^{k_0} D_{a,b}^0 \in V^{\partial \Sigma'},$
  and $\Delta_-^{N-k_0} \ _rD_{a,b}^{k_0}= \Delta_+ D_{a,b}^N \in V^{\partial \Sigma'}.$
\end{proof} 
We note after factoring in the relations of Lemma~\ref{lemma:leftright}, up to
smaller degree terms, the action of~$s$ on the elements $_lD_{a,b}^{k_0}$ is
like the action of $A^{2m(a+b)} \circ s \circ A^{-2m(a+b)}$ on the
$_r D_{a,b}^{k_0}.$ So that, up to smaller degree terms,
\[ s(_lD_{a,b}^{k_0})=\ _lD_{a+1,b+1}^{k_0} \equiv
A^{2m (a+b+2)}\  _rD_{a+1,b+1}^{k_0}=A^{4m} s(_rD_{a,b}^{k_0}). \]
Therefore, $\Delta_+(_lD_{a,b}^{k_0})=\Delta_{+,m}(_r D_{a,b}^{k_0})$ and
$\Delta_+^{k_0}(_lD_{a,b}^{k_0})=\Delta_{+,m}^{k_0}(_r D_{a,b}^{k_0}).$
In the end, modulo diagrams of smaller degree, we have:
\[
\begin{cases} \Delta_{+,m}^{k_0} \ _rD_{a,b}^{k_0} \in V^{\partial \Sigma'}
\\ \Delta_-^{N-k_0} \ _rD_{a,b}^{k_0} \in V^{\partial \Sigma'}.
\end{cases} \]
\begin{proof}[Proof of Proposition \ref{prop:pushout}]
  We will
  write $D \in V^{\partial \Sigma'}$ for a given diagram~$D$ to mean that
  $D \in V^{\partial \Sigma'}$ up to smaller degree diagrams.
  By Lemma~\ref{lemma:leftright}, we only want to show that
  $_rD_{a,b}^{k_0} \in V^{\partial \Sigma'}.$
  We have $\Delta_{+,m}^{k_0} (_rD_{a,b}^{k_0}) \in V^{\partial \Sigma'}$ and
  $\Delta_-^{N-k_0}(_rD_{a,b}^{k_0}) \in V^{\partial \Sigma'}.$
  We note that $\Delta_{+,m}=A^{4m+1}s-A^{-1}$ and $\Delta_-=-A^{-1}s+A$ commute, and that
  \[ \id_V=\frac{1}{A^{4m+2}-A^{-2}}\left( A^{-1}\Delta_{+,m}+A^{4m+1}\Delta_- \right). \]
  We conclude that
  $_rD_{a,b}^{k_0}=\id_V^{N}(_rD_{a,b}^{k_0})=
  \frac{1}{(A^{4m+2}-A^{-2})^N}\left( A^{-1}\Delta_{+,m}+A^{4m+1}\Delta_- \right)^N
  (_rD_{a,b}^{k_0}) \in V^{\partial \Sigma'},$
  as, after expanding, any term will contain either $\Delta_{+,m}^{k_0}$
  or $\Delta_-^{N-k_0}.$
\end{proof}
\begin{corollary}\label{coro:saucisson}
  Fix a sausage decomposition of $\Sigma.$ Any (arrowed) multicurve of the form
  \begin{center}
  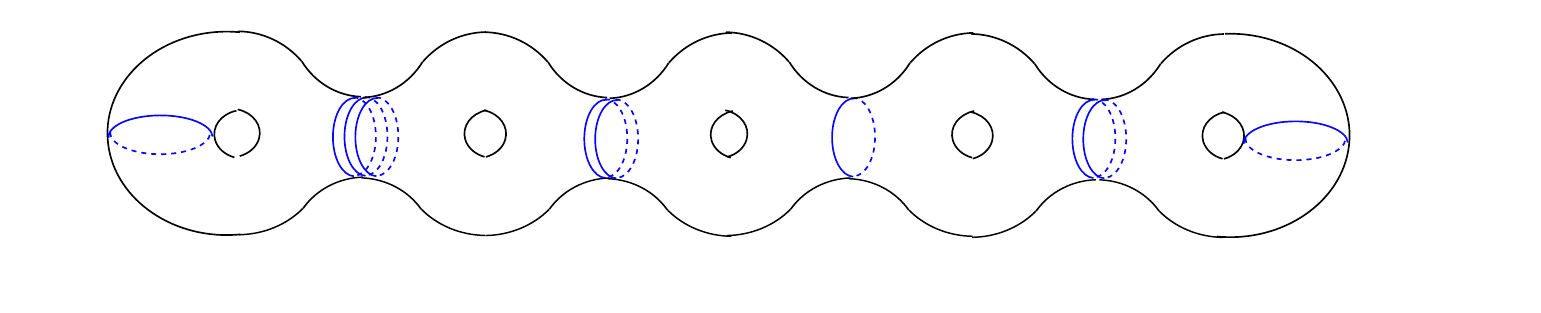
  \end{center}
  is actually a linear combination of curves of the same type with all $m_i=0,$
  and multicurves of smaller degree.
  In the above, arrows may be added in an arbitrary way.
\end{corollary}
It should be noted that by applying Corollary~\ref{coro:saucisson} we may
actually increase the complexity, as the separating curves are replaced by
non-separating ones, twice in number.
\begin{proof}
  Corollary~\ref{coro:saucisson} results of applying Proposition~\ref{prop:pushout}
  many times, pushing the groups of separating curves out step by step.
\end{proof}
From the results of Section~\ref{sec:dual_graph}, the skein module
$\Sk(\Sigma \times S^1)$ is spanned by multicurves of the type described in
Corollary~\ref{coro:saucisson}.
This may be used to
strengthen Corollary~\ref{coro:saucisson} to the following proposition,
which is the main result of the present section.
\begin{proposition}\label{prop:spanned_by_nsep_curves}
  The skein module $\Sk (\Sigma \times S^1)$ is spanned by arrowed trivial
  curves and arrowed non-separating simple closed curves.
\end{proposition}
\begin{proof}
  Thanks to Corollary~\ref{coro:saucisson}, we will only need to show to
  multicurves of the type described in the corollary and with $m_i=0$ and
  $n_1+n_2 \geqslant 2$ are linear combinations of multicurves of smaller degree.
  
  Let us first assume that $n_1$ or $n_2$ is at least $2.$ Without loss of
  generality, let us assume that $n_1\geqslant 2.$ Slightly adapting our previous
  notations, we define arrowed multicurves $D_{a,b}^k$ compatible with the sausage
  decomposition of~$\Sigma,$ that have $n_1-2$ extra curves on the left and $n_2$
  on the right. For example the multicurve we considered is of the form
  $D_{a,b}^1$ and looks like:
  \begin{center}
  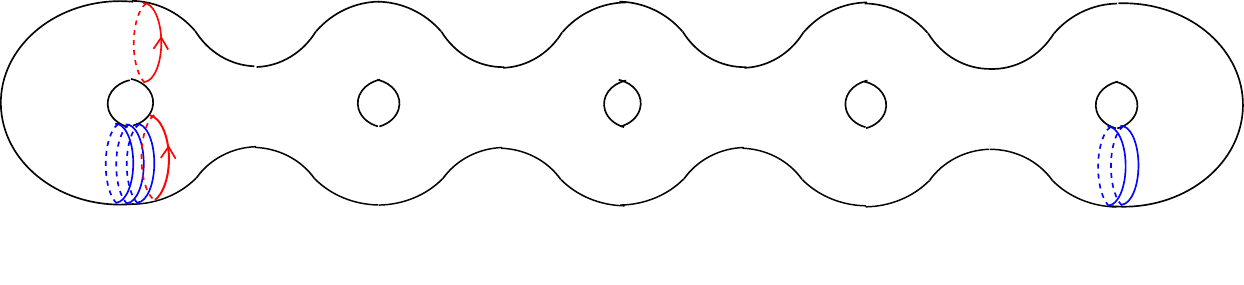
  \end{center}
  By comparison, the multicurve $D_{a,b}^{2g-1}$ would look like:
  \begin{center}
  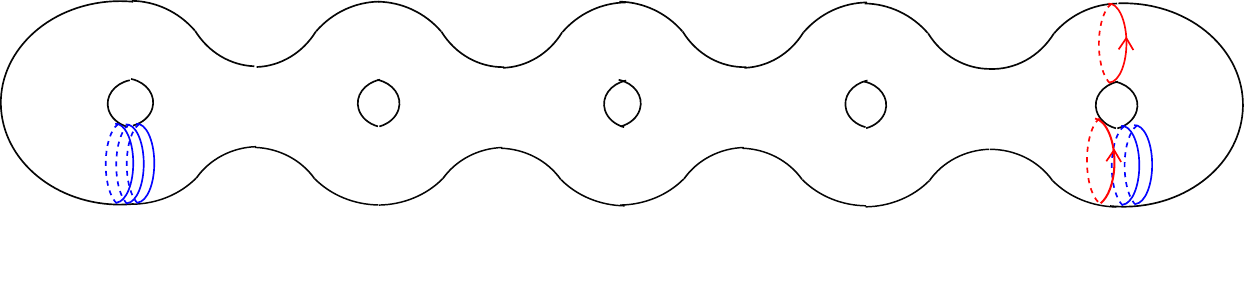
  \end{center}
  We will use the relations of Lemma~\ref{lemma:D_ab^k_relations} to relate the
  multicurves $D_{a,b}^k$ for $k \in \lbrace 0, \ldots , N \rbrace,$ while
  always leaving the ``extra'' blue curves unchanged. Note that we can always
  (up to multicurves of smaller degree) push all arrows from the blue curves
  to the red curves.
  Let $_lD_{a,b}^1$ denote
  the arrowed multicurve where the red curve in
  $D_{a,b}^1$ is put to the left of the blue curves.
  By Lemmas~\ref{lemma:D_ab^k_relations} and~\ref{lemma:leftright},
  we have $_lD_{a,b}^1\equiv A^{2(n_1-2)b}D_{a,b}^1$ and
  $\Delta_+ \ _lD_{a,b}^1\equiv 0.$ 
  The same computation as before thus gives $\Delta_{+,n_1-2} D_{a,b}^1 \equiv 0,$
  where $\Delta_{+,n_1-2}=A^{2n_1-3}s-A^{-1}.$
  
  Similarly, if $_rD_{a,b}^{2g-1}$ is the diagram obtained from the diagram
  $D_{a,b}^{2g-1}$ by putting the red arrowed curve to the right of the~$n_2$
  blue curves, we would have $\Delta_- \ _rD_{a,b}^{2g-1} \equiv 0$ and
  $_r D_{a,b}^{2g-1} \equiv A^{-2n_2 b}D_{a,b}^{2g-1}.$
  Thus $\Delta_{-,n_2} D_{a,b}^{2g-1} \equiv 0,$ where
  $\Delta_{-,n_2}=-A^{1-2n_2}s+A.$
  
  Thanks to Lemma \ref{lemma:D_ab^k_relations}, we have:
  \[ \begin{cases}
  \Delta_{+,n_1-2}D_{a,b}^1 \equiv 0
  \\ \Delta_{-,n_2}\Delta_-^{2g-2} D_{a,b}^1 \equiv
  \Delta_{-,n_2} \Delta_+^{2g-2} D_{a,b}^{2g-1} \equiv 0.
  \end{cases}\]
  Let us note that $\Delta_-, \Delta_{+,n_1-2}$ and $\Delta_{-,n_2}$ all commute.
  Moreover, let us note that
  \[ \mathrm{id}_V=\frac{1}{A^{2n_1-2}-A^{-2-2n_2}}\left( A^{-1-2n_2}
  \Delta_{+,n_1-2}+A^{2n_1-3}\Delta_{-,n_2} \right) \]
  and 
  \[ \mathrm{id}_V=\frac{1}{A^{2n_1-2}-A^{-2}}
  \left( A^{-1}\Delta_{+,n_1-2}+A^{2n_1-3}\Delta_- \right). \]
  When expanding the expression
  \begin{multline*}\mathrm{id}_V=\frac{1}{(A^{2n_1-2}-A^{-2-2n_2})(A^{2n_1-2}-A^{-2})^{2g-2}}
    \left( A^{-1-2n_2}\Delta_{+,n_1-2}+A^{2n_1-3}\Delta_{-,n_2} \right)
    \\ \circ \left( A^{-1}\Delta_{+,n_1-2}+A^{2n_1-3}\Delta_- \right)^{2g-2},
  \end{multline*}
  any term will contain either a factor $\Delta_{+,n_1-2}$ or a factor
  $\Delta_{-,n_2}\Delta_-^{2g-2}.$ Applying $\id_V$ to $D_{a,b}^1,$ we
  conclude that $D_{a,b}^1 \equiv 0.$
  This shows that as long as $n_1 \geqslant 2,$ the multicurve above is a
  linear combination of arrowed multicurves of smaller degree.
  
  Finally, in the remaining case where $n_1=n_2=1,$ we can fit the multicurve
  on a two-holed torus subsurface, so that the two boundary components of
  the two-holed torus are non separating in~$\Sigma.$
  Proposition~\ref{prop:2holed_torus_relation} then shows that the multicurve
  with $n_1=n_2=1$ is a linear combination of non-separating simple closed
  curves (the two boundary components of the two-holed torus).
\end{proof}

\section{Elimination of arrows}
By the previous section, the skein module $\Sk (\Sigma \times S^1)$ is spanned
by all arrowed multicurves whose underlying multicurve is either a non-separating
curve or a trivial curve. We will now study the ``vertical'' part of those curves,
that is, relate elements of $\Sk(\Sigma \times S^1)$ that differ only by the
number of arrows we put on them.
We will treat the cases of non-separating curves and of the trivial curve separately.
\label{sec:elimin_arrows}

\subsection{Arrows on non-separating curves}
\label{sec:arrows_nsep}
We have:
\begin{proposition}\label{prop:arrows_non_sep}
  Let $\gamma$ be a non-separating simple closed curve, with some choice of
  orentation, and for $n\in \Z$ let~$\gamma_n$ be the arrowed curve~$\gamma$
  with~$n$ arrows in the direction of~$\gamma.$
  Then, for any $n\in \Z,$ we have $\gamma_n=\gamma_{n-2} \in \Sk( \Sigma \times S^1).$
\end{proposition}
Based on the above proposition, to span $\Sk(\Sigma \times S^1),$ it is
sufficient to consider non-separating curves with $0$ or $1$ arrow.
Moreover, for non curves with~$1$ arrow the direction of the arrow can be
chosen arbitrarily.
\begin{proof}
  We begin by observing that by using relation $R_5$ of
  Section~\ref{sec:arrowed_diagrams} we have:
  \begin{center}
  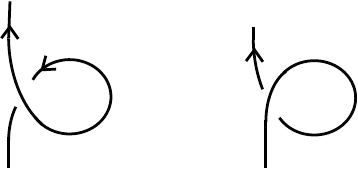
  \end{center}
  Using Kauffman relations $K_1$ and $K_2,$ this gives the \textit{arrow-shift relation}:
  \begin{center}
  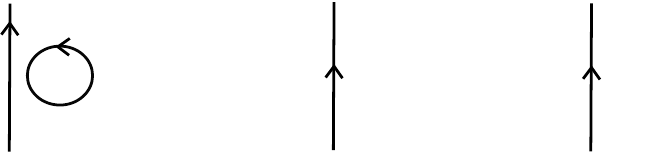
  \end{center}
  Hence, adding one trivial curve with one arrow in the direct orientation,
  to the diagram $\gamma_n$, yields the linear combination
  $-A^2\gamma_{n+1}-A^4\gamma_{n-1}$.
  
  On the other hand, the equality
  \begin{center}
  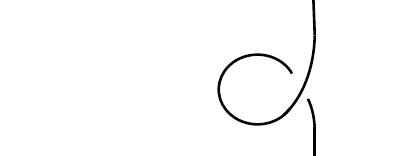
  \end{center}
  gives, this time, that the same diagram obtained by adding one trivial
  curve with one arrow in the direct orientation to $\gamma_n$ equals
  $-A^2\gamma_{n-1}-A^4\gamma_{n+1}$. The equality between these two expressions
  is equivalent to $\lbrace 2 \rbrace (\gamma_{n+1}-\gamma_{n-1})=0.$
  This proves the proposition.
\end{proof}

\subsection{Arrows on the trivial curve}
\label{sec:arrow_trivial}
We now turn to the case of the trivial curve.
To set things up, for $n \in \Z$ let the arrowed curve~$S_n$ be the trivial
curve with~$n$ arrows in the positive direction. If we fix a sausage
decomposition of~$\Sigma,$ then~$S_n$ also corresponds to the curve
$D_{n,0}^0$ defined in Section~\ref{sec:sausage}.
It also corresponds to the curve~$D_{-n,0}^{2g}.$
We introduce a
last operator~$\theta$ on the vector space~$V$ formally spanned by the~$S_n$ by
$\theta(S_n)=S_{-n}.$ This operator will be treated similarly as we treated
the shift operator~$s$ (which we recall is defined by $s(S_n)=S_{n+1}$ ),
and the operators $\Delta_+=As-A^{-1}$ and $\Delta_-=-A^{-1}s+A$ in
Section~\ref{sec:sausage}: they are
only defined as linear operators on~$V,$ but we will use them to write
relations in $\Sk(\Sigma \times S^1),$
as in Section~\ref{sec:sausage}.

In this context, as there are no ``extra'' curves that act as barriers here,
the relations given by Lemma~\ref{lemma:D_ab^k_relations} simply read
\[ \Delta_- D_{a,b}^k=\Delta_+ D_{a,b}^{k+1} \]
for any $a,b \in \Z$ and any $0 \leqslant k \leqslant 2g-3.$
Thanks to these relations, we can show that the arrowed curves~$S_n$
satisfy the following system of relations:
\begin{proposition}\label{prop:S_n^k_system}
  The multicurves $S_n \in \Sk (\Sigma \times S^1)$ satisfy:
  \begin{enumerate}
  \item $\forall n\geqslant 1, \quad
   A^{-n-2}S_n-A^{n+2}S_{-n} \in \mathrm{Span}_{\QQ(A)}\left(S_0,\ldots ,S_{n-1}\right).$
  \item $\Delta_-^{2g}(S_n)=(\Delta_+^{2g} \circ \theta) (S_n).$
  \end{enumerate}
\end{proposition}
\begin{proof}
  The second point of the proposition results from our remarks above. Indeed,
  from Lemma \ref{lemma:D_ab^k_relations}, we get for any $n \in \Z$
  \[ \Delta_+^{2g} D_{n,0}^{2g}=\Delta_- D_{n,0}^0 \]
  which recalling that $D_{n,0}^0=S_n$ and $D_{n,0}^{2g}=S_{-n}$ gives
  exactly point (2) of the proposition.
  
  Thus we just need to prove point (1).
  Recall the arrow-shift relation obtained in the preceding section:
  \begin{center}
  \input{Trivial_curve_rel3.pdf_tex}
  \end{center}
  Applying this relation to the diagram 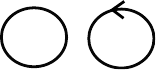
  will give us that $A^{-2}S_1=A^4S_{-1},$ which is the $n=1$ case of (1).
  We proceed to prove (1) by induction on $n.$
  
  Assume that (1) has been etablished for some $n.$ Using the arrow-shift relation, we have:
  \begin{center}
  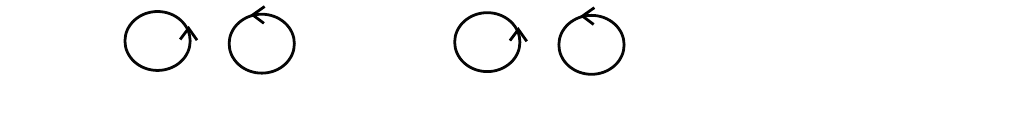
  \end{center}
  By induction hypothesis and the arrow-shift relation, the left hand side
  is in $\mathrm{Span}_{\QQ(A)}(S_0,\ldots ,S_{n}).$
  Thus $A^{-n}S_{n+1}-A^{n+6}S_{-n-1} \in \mathrm{Span}_{\QQ(A)}(S_0,\ldots ,S_{n}),$
  and by induction, (1) holds for all $n\geqslant 1.$
\end{proof}
We will use this system to prove the following proposition, that shows that
the subspace of $\Sk(\Sigma \times S^1)$ spanned by arrowed trivial curves
is finite dimensional:
\begin{proposition}\label{prop:arrows_sep}
  The subspace of $\Sk (\Sigma \times S^1)$ spanned by the arrowed curves
  $(S_n)_{n\in \Z,}$ is actually spanned by the curves $S_n$ for $n=0,\ldots,2g.$
\end{proposition}
\begin{proof}
  Let us take a closer look at the equation 
  \[ \Delta_-^{2g}(S_n)=(\Delta_+^{2g} \circ \theta)(S_n). \]
  We expand both expressions, using that $\Delta_-=-A^{-1}s+A$ and
  $\Delta_+=As-A^{-1}.$ We get:
  \[ \underset{k=0}{\overset{2g}{\sum}}\binom{n}{k}(-1)^k A^{2g-2k}S_{n+k}=
  \underset{k=0}{\overset{2g}{\sum}}\binom{n}{k}(-1)^k A^{2k-2g}S_{-n-k}. \]
  Let us now assume $n \geqslant 1.$ Extracting the terms in~$S_l$ with $|l|$ maximal
  from both sides, we get $A^{-2g}S_{n+2g}\equiv A^{2g}S_{-n-2g}$ modulo
  $\mathrm{Span}(S_{-n-2g+1} ,\ldots ,S_{n+2g-1}).$
  Now if $\equiv$ is equality modulo $\mathrm{Span}(S_{-n-2g+1} ,\ldots ,S_{n+2g-1}),$
  using Equation (1) of Proposition~\ref{prop:S_n^k_system}, we get the invertible system
  \[ \begin{cases}A^{-2g}S_{n+2g}-A^{2g}S_{-n-2g} \equiv 0
  \\ A^{n+2g+2}S_{n+2g}-A^{-n-2g-2}S_{-n-2g} \equiv 0
  \end{cases} \]
  which, using again Equation (1) of Proposition~\ref{prop:S_n^k_system},
  implies that , for any $n\geqslant 1,$
  \[ S_{n+2g},S_{-n-2g} \in \mathrm{Span}(S_0,\ldots ,S_{n+2g-1}). \]
  We remark that using only Equation (1), by induction we can show that when
  $|n| \leqslant 2g$ then $S_n \in \mathrm{Span}(S_0,\ldots,S_{2g}).$
  From there we can use the above result to inductively deduce that, for
  any $n \in \Z,$ the element $S_n$ is in $\mathrm{Span}(S_0,\ldots ,S_{2g}).$
\end{proof}

\section{Relating non-separating curves}
\label{sec:rel_nsep_curves}
In this last section, we conclude the proof of Theorem~\ref{thm:basis_skein}.
Thanks to Section~\ref{sec:reducing_degree}, we know that $\Sk(\Sigma \times S^1)$
is spanned by arrowed non-separating curves and arrowed trivial curves and
thanks to Section~\ref{sec:elimin_arrows}, we know that we need only
non-separating curves with~$0$ or~$1$ arrow and trivial curves with at
most~$2g$ arrows to span $\Sk(\Sigma \times S^1).$
Thus to prove Theorem~\ref{thm:basis_skein}, the only missing ingredient is to
prove that two non-separating curves~$\gamma$ and~$\gamma'$
(both with~$0$ or~$1$ arrow) such that $[\gamma]=[\gamma'] \in H_1(\Sigma, \Z/2)$
represent the same element in $\Sk(\Sigma \times S^1),$ which is what we
prove in this section.

\subsection{Action of Dehn twists on the fundamental group}
\label{sec:mapping_class_group}
\begin{figure}[hb]
\centering
\def \svgwidth{0.6\columnwidth}
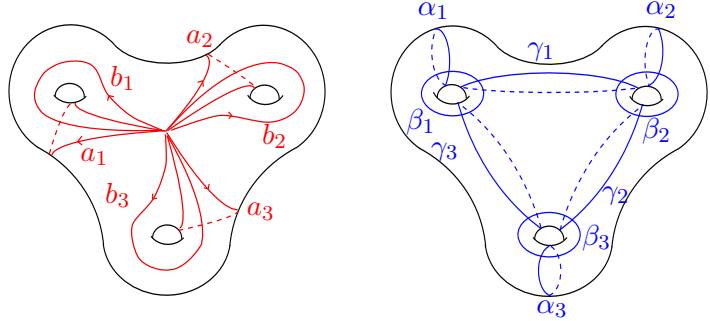
\caption{On the left, the standard generators of $\pi_1(\Sigma)$, on the right,
the Lickorish generators of $\MCG (\Sigma)$, drawn for $\Sigma$ a genus~$3$ surface.}
\label{fig:lickorish}
\end{figure}
In this section, we will set a few notations for the fundamental group and
mapping class group of~$\Sigma,$ and perform some elementary computations,
that we will need in Section \ref{sec:equiv_classes}.

Here, let $\Sigma$ be a closed compact oriented surface of genus $g.$
For elements $a,b \in \pi_1(\Sigma),$ we adopt the convention that $a\cdot b$
is the path obtained by following first the oriented loop $a$ then following
the loop~$b.$
Let $a_1,b_1,\ldots,a_g,b_g$ be the standard generators of $\pi_1(\Sigma),$ as
shown on the left of Figure~\ref{fig:lickorish}, so that $\pi_1(\Sigma)$ is the group
\[ \pi_1(\Sigma)=\langle a_1,b_1,\ldots, a_g,b_g\rangle/_{[a_1,b_1][a_2,b_2]\ldots [a_g,b_g]=1}. \]
We also introduce three families of simple closed curves on $\Sigma:$ the curves
$\alpha_i,\beta_i,\gamma_i$ represented on the right of Figure~\ref{fig:lickorish}.
One can easily check that the curve $\alpha_i$ (resp. $\beta_i$ and $\gamma_i$)
represents the free homotopy class $[a_i]$
(resp. $[b_i]$ and $[a_{i+1}^{-1}b_i a_i b_i^{-1}]$).
The Dehn twists along the curves $\alpha_i,\beta_i,\gamma_i$ form the well-known
Lickorish generators of the mapping class group $\MCG (\Sigma):$
\begin{theorem}\cite{Lickorish}
  The $3g-1$ Dehn twists $\tau_{\alpha_i},\tau_{\beta_i}$ for
  $1\leqslant i \leqslant g$ and $\tau_{\gamma_i}$ for
  $1\leqslant i \leqslant g-1$ generate $\MCG (\Sigma).$
\end{theorem}
For the use of the next section, let us collect here
a few formulas expressing the action of the Dehn twists
$\tau_{\alpha_i},\tau_{\beta_i},\tau_{\gamma_i}$ on $\pi_1(\Sigma).$
\begin{lemma}\label{lemma:dehn_twist_formulas}
  Let $\varepsilon \in \lbrace \pm 1 \rbrace.$
  \begin{itemize}
  \item[-]The map $\tau_{\alpha_i}^{\varepsilon}$ sends $b_i$ to
    $b_i a_i^{\varepsilon},$ and leaves all other generators of
    $\pi_1(\Sigma)$ invariant.
  \item[-]The map $\tau_{\beta_i}^{\varepsilon}$ sends $a_i$ to
    $a_i b_i^{-\varepsilon},$ and leaves all other generators invariant.
  \item[-]The map $\tau_{\gamma_i}^{\varepsilon}$ sends $b_i$ to
    $\gamma_i^{\varepsilon} b_i,$ sends $a_{i+1}$ to
    $\gamma_i^{\varepsilon} a_{i+1} \gamma_i^{-\varepsilon},$ sends $b_{i+1}$ to
    $b_{i+1}\gamma_i^{-\varepsilon},$ and leaves all other (including $a_i$)
    generators of $\pi_1(\Sigma)$ invariant.
  \end{itemize}
\end{lemma}
The proof of the lemma consists of homotopying the images of generators by the
Dehn twists, and is left as exercise for the reader.

\subsection{An equivalence relation on the set of simple closed curves}
\label{sec:equiv_classes}
For $\gamma,\delta$ two simple closed curves on $\Sigma,$ let $i(\gamma,\delta)$
be the geometric intersection number of $\gamma$ and $\delta.$ Let also
$[\gamma]\in H_1(\Sigma,\Z/2)$ be the $\Z/2$-homology class of~$\gamma.$
Finally, for any simple closed curve $\gamma$ on $\Sigma,$ let $\tau_{\gamma}$
denote the Dehn twist along~$\gamma.$
We find some elementary equalities between different simple closed curves
in $\Sk (\Sigma \times S^1):$
\begin{proposition}\label{prop:dehn_twist_rel}
  Let $\gamma$ and $\delta$ be two simple closed curves on $\Sigma,$ viewed
  as elements of $\Sk (\Sigma \times S^1).$ Then:
  \begin{itemize}
  \item[-]If $i(\gamma,\delta)=1,$ then $\gamma=\tau_{\delta}^2(\gamma).$
  \item[-]If $i(\gamma,\delta)=2,$ then $\gamma=\tau_{\delta}(\gamma).$
  \end{itemize}
  Moreover, those relations stay true after decorating $\gamma$ and
  $\tau_{\delta}^{n}(\gamma)$ with the same number of arrows.
\end{proposition}
\begin{proof}
  If $i(\gamma,\delta)=1,$ a neighborhood of $a \cup b$ in $\Sigma$ is a
  one-holed torus. Moreover, $\tau_{\delta}(\gamma)$ and $\delta$ also
  intersects once. Looking at the Kauffman resolution of the equation
  $\tau_{\delta}(\gamma)\cdot \delta=\delta \cdot \tau_{\delta}(\gamma),$ we have:
  \begin{center}
  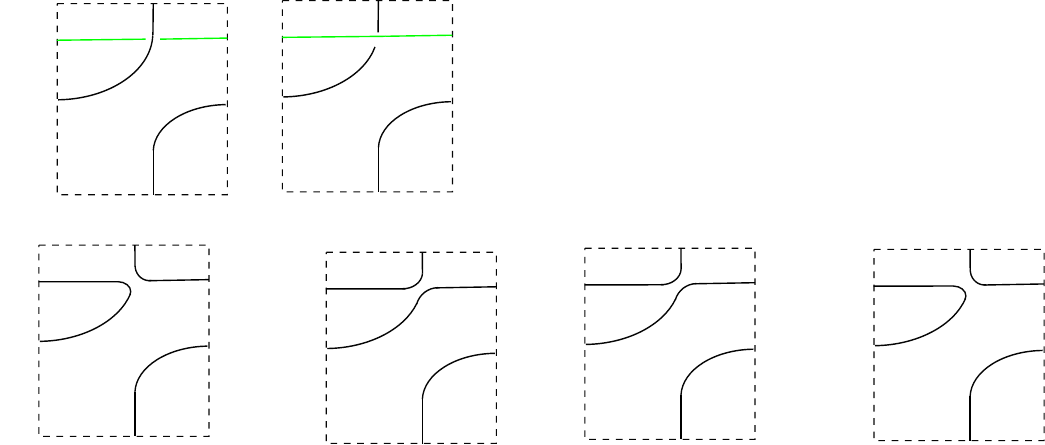
  \end{center}
  which gives us that $(A-A^{-1})\gamma=(A-A^{-1})\tau_{\delta}^2(\gamma),$ and
  thus $\gamma=\tau_{\delta}^2(\gamma)$ as we work over $\QQ(A)$ coefficients.
  
  If $i(\gamma,\delta)=2,$ if the algebraic intersection of $a$ and $b$ is $0,$
  then a neighborhood of $a \cup b$ in $\Sigma$ is a $4$-holed sphere.
  Otherwise, a neighborhood of $a \cup b$ is a $2$-holed torus. Let us assume
  the former. Let $\gamma'$ be the $1/2$ fractional Dehn twist of $\gamma$ along
  $\delta.$ The equation $\gamma'\cdot \delta=\delta \cdot \gamma'$ reads:
  \begin{center}
  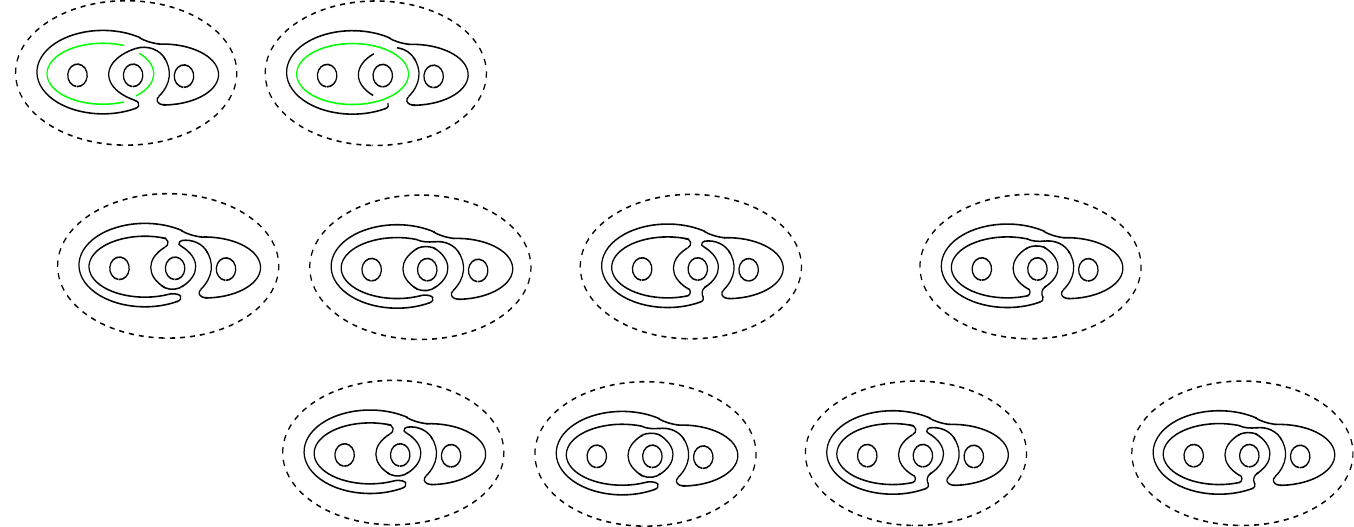
  \end{center}
  which implies that $(A^2-A^{-2})\gamma=(A^2-A^{-2})\tau_{\delta}(\gamma),$ and
  thus $\gamma=\tau_{\delta}(\gamma).$
  
  The case where the algebraic intersection of $a$ and $b$ is $\pm 2$ is
  fairly similar and left to the reader. 
\end{proof}
This computation leads us to define an equivalence relation on non-separating
simple closed curves on~$\Sigma.$
\begin{definition}\label{def:equiv_relation}
  On the set of non-separating simple closed curves on $\Sigma,$ let $\sim$
  be the equivalence relation generated by:
  \begin{itemize}
  \item[-]If $\gamma, \delta$ are simple closed curves such that
    $i(\gamma,\delta)=1,$ then $\gamma \sim \tau_{\delta}^2(\gamma).$
  \item[-]If $i(\gamma,\delta)=2,$ then $\gamma \sim \tau_{\delta}(\gamma).$
  \end{itemize}
\end{definition}
The proof of Theorem~\ref{thm:basis_skein} now reduces to the following proposition.
\begin{proposition}\label{prop:equiv_classes}
  Let $\gamma$ and $\gamma'$ be two simple closed non-separating curves
  on $\Sigma.$ Then $\gamma \sim \gamma'$ if and only if
  $[\gamma]=[\gamma'] \in H_1(\Sigma ,\Z/2).$
\end{proposition}
  Although we think it is likely that an appropriate use of the mapping class group literature could lead to a short proof of Proposition \ref{prop:equiv_classes}, we were unable to find a statement that would directly apply. Instead, we had to resort to a brute force proof.
\begin{proof}[Proof of Proposition \ref{prop:equiv_classes}]
  Notice that by definition, if $\gamma \sim \gamma'$ then there is a mapping
  class group element $\sigma \in \MCG(\Sigma)$ such that
  $\gamma'= \sigma(\gamma).$ Moreover, if $\gamma$ and $\gamma'$ are related
  by a generating relation as in Definition \ref{def:equiv_relation}, then
  clearly $[\gamma]=[\gamma'].$ The direct implication follows.
  
  Next we remark that if $\gamma \sim \delta,$ then for any
  $\sigma \in \MCG(\Sigma),$ we have that $\sigma(\gamma) \sim \sigma(\delta).$
  Indeed, if for example $\gamma'=\tau_{\delta}^2(\gamma)$ and $i(\delta,\gamma)=1,$ then
  $\sigma(\gamma')=\sigma \circ \tau_{\delta}^2 \circ \sigma^{-1} (\sigma(\gamma))
  =\tau_{\sigma(\delta)}(\sigma(\delta)).$
  Moreover $i(\sigma(\delta)),\sigma(\gamma))=i(\delta,\gamma)=1,$ so
  $\sigma(\gamma')\sim \sigma(\gamma).$ The same is true for the other generating
  relations, and the general case follows by transitivity.
  
  Let us now introduce the finite $\mathcal{F}$ of elements of $\pi_1(\Sigma)$
  of the type:
  \[ a_1^{\varepsilon_1}b_1^{-\delta_1} \ldots a_g^{\varepsilon_g} b_g^{-\delta_g}, \]
  where the $\varepsilon_i,\delta_i$ are elements of $\lbrace 0,1\rbrace,$
  non all-zero. Notice that $\mathcal{F}$ contains exactly one element in each
  non-zero homology class of $H_1(\Sigma,\Z/2).$ Moreover, all of these loops
  actually represent simple closed curves on $\Sigma,$ which are just connected
  sums of simple closed curves $a_i^{\varepsilon_i} b_i^{-\delta_i}.$
  Let also $\mathcal{G}$ denote the set of Lickorish generators:
  $\mathcal{G}=\lbrace \tau_{\alpha_i}, \tau_{\beta_i},\tau_{\gamma_i} \rbrace.$
  By the above discussion, Proposition~\ref{prop:equiv_classes} will follow
  once we prove:
  \begin{lemma}\label{lemma:equiv_classes}
    For any simple closed curve
    $c=a_1^{\varepsilon_1} b_1^{-\delta_1}a_2^{\varepsilon_2}b_2^{-\delta_2}
    \ldots a_g^{\varepsilon_g}b_g^{-\delta_g} \in \mathcal{F},$
    for any Dehn twist
    $\tau \in \mathcal{G}=\lbrace \tau_{\alpha_i},\tau_{\beta_i},\tau_{\gamma_i}\rbrace,$
    we have that $\tau(c)$ and $\tau^{-1}(c)$ are equivalent to elements of $\mathcal{F}.$
  \end{lemma}
  \begin{proof}[Proof of Lemma \ref{lemma:equiv_classes}]
    Let us treat first the case of the generators $\tau_{\alpha_i}.$ For $c \in \mathcal{F},$
    let us write $c=w a_i^{\varepsilon_i} b_i^{-\delta_i} z$ where $w,z$ are expressed in
    generators of $\pi_1(\Sigma)$ different than $a_i,b_i.$
    If $\delta_i=0,$ then we have $\tau_{\alpha_i}^{\pm 1}(c)=c,$ so the
    $\tau_{\alpha_i}^{\pm 1}(c)$ are equivalent to elements of $\mathcal{F}.$
    So let us assume $\delta_i=1.$
    If $\delta_i=1,$ then $i(c,\alpha_i)=1,$ and thus
    $\tau_{\alpha_i}(c) \sim \tau_{\alpha_i}^{-1}(c).$
    So it is sufficient to prove that one of the two is equivalent to an
    element of $\mathcal{F}.$
    Let $\mu \in \lbrace \pm 1 \rbrace.$ By the formulas in
    Lemma~\ref{lemma:dehn_twist_formulas}, we have that
    $\tau_{\alpha_i}^{\mu}(c)=w a_i^{\varepsilon_i} a_i^{-\mu} b_i^{-1} z.$
    Depending on the value of $\varepsilon_i,$ we see that either
    $\tau_{\alpha_i}(c)$ or $\tau_{\alpha_i}^{-1}(c)$ is an element of $\mathcal{F},$
    so both are equivalent to elements of $\mathcal{F}.$
    
    Working with the generators $\tau_{\beta_i}$ is similar: still writing
    $c=w a_i^{\varepsilon_i} b_i^{-\delta_i} z,$ we have that if $\varepsilon_i=0$
    then $\tau_{\beta_i}^{\pm 1}(c)=c,$ which is an element of $\mathcal{F}$ already.
    Else $i(\beta_i,c)=1,$ so that $\tau_{\beta_i}(c) \sim \tau_{\beta_i}^{-1}(c)$
    and moreover we find that one of the $\tau_{\beta_i}^{\pm 1}(c)$ is an
    element of $\mathcal{F}.$
    We finally turn to the case of the Lickorish generators $\tau_{\gamma_i}.$
    This time, let us write
    $c=w tz =w a_i^{\varepsilon_i} b_i^{-\delta_i} a_{i+1}^{\varepsilon_{i+1}}b_{i+1}^{-\delta_{i+1}}z.$
    There are $16$ possibilities for the middle word $t,$ which we subdivide
    into $3$ categories, depending on the geometric intersection number with $\gamma_i:$
    \begin{itemize}
    \item[-]For $t=1, a_i, b_i^{-1}b_{i+1}^{-1}, a_i b_i^{-1}b_{i+1}^{-1},
      b_i^{-1}a_{i+1}b_{i+1}^{-1},a_i b_i^{-1} a_{i+1} b_{i+1}^{-1}$
      we have that $i([t],\gamma_i)=0.$
    \item[-]For $t=b_i^{-1},a_i b_i^{-1}, b_i^{-1}a_{i+1}, a_i b_i^{-1} a_{i+1},
      b_{i+1}^{-1}, a_i b_{i+1}^{-1}, a_{i+1} b_{i+1}^{-1}, a_i a_{i+1} b_{i+1}^{-1},$
      we have that $i([t],\gamma_i)=1.$
    \item[-]For $t=a_{i+1}$ or $a_i a_{i+1},$ we have that $i([t],\gamma_i)=2.$
    \end{itemize}
    Again, if $i([t],\gamma_i)=0,$ then $\tau_{\gamma_i}^{\pm 1}(c)=c$ and we have
    nothing to prove. Moreover if $i([t],\gamma_i)=2,$ then by definition of the
    relation $\sim,$ we have that $\tau_{\gamma_i}^{\pm 1}(c) \sim c.$
    
    \begin{figure}[hb]
    \centering
    \def \svgwidth{1.0\columnwidth}
    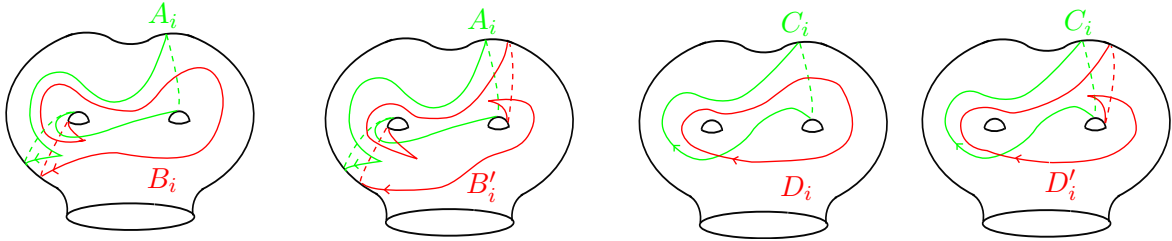
    \caption{The curves $A_i,B_i,B_i',C_i,D_i,D_i'$ in the genus $2$ subsurface
      with boundary the curve corresponding to the free homotopy class of
      $[a_i,b_i][a_{i+1},b_{i+1}]\in \pi_1(\Sigma).$}
    \label{fig:extra_curves}
    \end{figure}
    It remains to show that if $i([t],\gamma_i)=1$ then the elements
    $\tau_{\gamma_i}^{\pm 1}(c)$ are equivalent to elements of $\mathcal{F}.$
    As before, those two elements are equivalent, so we only have to find that
    one of them is equivalent to an element of $\mathcal{F}$ in each case.
    We proceed to show this in the remaining $8$ cases:
    
    \noindent
    1) If $t=b_i^{-1}$ then 
    $\tau_{\gamma_i}(c)=w b_i^{-1} \gamma_i^{-1} z=w a_i^{-1}b_i^{-1}a_{i+1} z.$
    But then we have that $i(\tau_{\gamma_i}(c),\alpha_i)=1,$ so that
    $\tau_{\gamma_i}(c)\sim \tau_{\alpha_i}^{-2}\circ \tau_{\gamma_i}(c).$
    But we compute that:
    \[ \tau_{\alpha_i}^{-2}\circ \tau_{\gamma_i}(c)=w a_i b_i^{-1} a_{i+1} z \]
    which is an element of $\mathcal{F}.$
    
    \noindent
    2) If $t=a_i b_i^{-1}$ then
    $\tau_{\gamma_i}(c)=w a_i b_i^{-1} \gamma_i^{-1} z=w b_i^{-1} a_{i+1} z \in \mathcal{F}.$
    
    \noindent
    3) If $t=b_i^{-1}a_{i+1}$ then 
    $\tau_{\gamma_i}^{-1}(c)=w b_i^{-1} a_{i+1} \gamma_i z=w a_i b_i^{-1} z$
    which is an element of $\mathcal{F}.$
    
    \noindent
    4) If $t=a_i b_i^{-1} a_{i+1}$ then 
    $\tau_{\gamma_i}^{-1}(c)=w a_i b_i^{-1} a_{i+1} \gamma_i z=w a_i^2 b_i^{-1} z.$
    Thus $i(\alpha_i, \tau_{\gamma_i}^{-1}(c))=1$ so that
    $\tau_{\gamma_i}^{-1}(c)\sim \tau_{\alpha_i}^2\circ \tau_{\gamma_i}^{-1}(c).$
    We compute that
    \[ \tau_{\alpha_i}^2\circ \tau_{\gamma_i}^{-1}(c)=w b_i^{-1}z \]
    is an element of $\mathcal{F}.$
    
    \noindent
    5) If $t=b_{i+1}^{-1}$ then
    $\tau_{\gamma_i}^{-1}(c)=
    w \gamma_i^{-1} b_{i+1}^{-1} z=w b_i a_i^{-1} b_i^{-1} a_{i+1} b_{i+1}^{-1}z.$
    We find that $i(\tau_{\gamma_i}^{-1}(c), \beta_i)=1,$ so that
    $\tau_{\gamma_i}^{-1}(c) \sim \tau_{\beta_i}^{-2} \circ \tau_{\gamma_i}(c).$
    We compute that:
    \[ c'=\tau_{\beta_i}^{-2} \circ \tau_{\gamma_i}^{-1}(c)=
    w b_i^{-1} a_i^{-1} b_i^{-1} a_{i+1} b_{i+1}^{-1} z. \]
    But now $i(c',\alpha_i)=2,$ so that $c\sim c' \sim \tau_{\alpha_i}^{-1}(c').$
    We compute that
    \[ \tau_{\alpha_i}^{-1}(c')=w a_i b_i^{-1} a_{i+1} b_{i+1}^{-1} z \]
    is an element of $\mathcal{F}.$
    
    \noindent
    6) If $t=a_{i+1} b_{i+1}^{-1}$ then
    $\tau_{\gamma_i}^{-1}(c)=w \gamma_i^{-1}a_{i+1} b_{i+1}^{-1} z=
    w b_i^{-1} a_i^{-1} b_i b_{i+1}^{-1}z.$
    Similarly to the previous case, we find that 
    \[ c\sim \tau_{\beta_i}^{-2}\circ \tau_{\gamma_i}^{-1}(c)=
    c' \sim \tau_{\alpha_i}^{-1}(c')=w a_i b_i^{-1} b_{i+1}^{-1} z \]
    which is an element of $\mathcal{F}.$
    
    \noindent
    7) If $t=a_i b_{i+1}^{-1}$ then
    $\tau_{\gamma_i}(c)=w a_i \gamma_i b_{i+1}^{-1} z=
    w a_i a_{i+1}^{-1} b_i a_i b_i^{-1} b_{i+1}^{-1}z=c'.$
    Let us introduce $A_i=a_i a_{i+1}^{-1} b_i$ and $B_i=a_i b_i^{-1} b_{i+1}^{-1}.$
    By abuse of notation, we also write $A_i$ and $B_i$ for the simple closed
    curves corresponding to the free homotopy classes $[A_i],[B_i]$.
    Those simple closed curves are represented on Figure~\ref{fig:extra_curves}.
    We see on the Figure that $i(A_i,B_i)=1,$ and also that
    $\tau_{A_i}(B_i)=A_i^2 B_i.$ It is moreover clear that $\tau_{A_i}$
    leaves all the generators $a_j,b_j$ with $j$ not $i$ or $i+1$ invariant.
    Thus we have that 
    \[ \tau_{\gamma_i}(c)=c' \sim \tau_{A_i}^{-2}(c')=
    w A_i^{-1}B_i z=w (b_i^{-1} a_{i+1}) (b_i^{-1} b_{i+1}^{-1}) z=c''. \]
    Now calling $C_i$ and $D_i$ the simple closed curves corresponding to the free
    homotopy classes $[b_i^{-1} a_{i+1}]$ and $[b_i^{-1} b_{i+1}^{-1}].$
    Again we have that $i(C_i,D_i)=1,$ and that $\tau_{C_i}(D_i)=C_i^2 D_i,$
    with $\tau_{C_i}$ leaving the $a_j,b_j$ with $j$ not $i$ or $i+1$ invariant. So,
    \[ \tau_{\gamma_i}(c) \sim c'' \sim \tau_{C_i}^{-2}(c'')=
    w C_i^{-1} D_i z=w a_{i+1}^{-1} b_{i+1}^{-1}z. \]
    Finally, this last element is equivalent to $w a_{i+1} b_{i+1}^{-1}z,$ an
    element of $\mathcal{F},$ using the square of the Dehn twist along $\alpha_{i+1}.$
    
    \noindent
    8) Finally, if $t=a_i a_{i+1} b_{i+1}^{-1}$ then
    $\tau_{\gamma_i}(c)=w a_i  \gamma_i a_{i+1} b_{i+1}^{-1} z=
    w a_i a_{i+1}^{-1} b_i a_i b_i^{-1} a_{i+1} b_{i+1}^{-1}z=c'.$
    Let this time $A_i=a_i a_{i+1}^{-1} b_i$ and $B_i'=a_i b_i^{-1} a_{i+1} b_{i+1}^{-1}.$
    We still have that $i(A_i, B_i')=1,$ and that $\tau_{A_i}(B_i')=A_i^2 B_i',$
    so that similarly to the previous case we get:
    \[ c' \sim w A_i^{-1} B_i'z=w b_i^{-1} a_{i+1} b_i^{-1} a_{i+1} b_{i+1}^{-1}z.\]
    Setting again $C_i=b_i^{-1} a_{i+1}$ and $D_i'=b_i^{-1} a_{i+1} b_{i+1}^{-1},$
    we still check that $i(C_i, D_i')=1$ and thus that
    \[ c' \sim \tau_{C_i}^{-2}(c')=c''= w C_i^{-1} D_i'z= w b_{i+1}^{-1} z \]
    which is an element of $\mathcal{F}.$
  \end{proof}
  Lemma~\ref{lemma:equiv_classes} now being established,
  Proposition~\ref{prop:equiv_classes}
  follows: by induction, for any word in the Lickorish generators $w$ and any
  element $s$ of $\mathcal{F},$ there is $s'\in \mathcal{F}$ so that $w(s)\sim s'.$
  As any non-separating simple closed curve is of the form $w([a_1])$ for some
  $w\in \MCG (\Sigma),$ any non-separating simple closed curve is equivalent
  to an element of $\mathcal{F}.$
\end{proof}

\bibliographystyle{hamsplain}
\bibliography{biblio}
\end{document}

%% file: kauffman.pdf_tex
\begingroup%
  \makeatletter%
  \providecommand\color[2][]{%
    \errmessage{(Inkscape) Color is used for the text in Inkscape, but the package 'color.sty' is not loaded}%
    \renewcommand\color[2][]{}%
  }%
  \providecommand\transparent[1]{%
    \errmessage{(Inkscape) Transparency is used (non-zero) for the text in Inkscape, but the package 'transparent.sty' is not loaded}%
    \renewcommand\transparent[1]{}%
  }%
  \providecommand\rotatebox[2]{#2}%
  \ifx\svgwidth\undefined%
    \setlength{\unitlength}{183.60766602bp}%
    \ifx\svgscale\undefined%
      \relax%
    \else%
      \setlength{\unitlength}{\unitlength * \real{\svgscale}}%
    \fi%
  \else%
    \setlength{\unitlength}{\svgwidth}%
  \fi%
  \global\let\svgwidth\undefined%
  \global\let\svgscale\undefined%
  \makeatother%
  \begin{picture}(1,0.3520064)%
    \put(0,0){\includegraphics[width=\unitlength]{kauffman.pdf}}%
    \put(0.27969297,0.2468334){\color[rgb]{0,0,0}\makebox(0,0)[lb]{\smash{$=A$}}}%
    \put(0.60852019,0.2456939){\color[rgb]{0,0,0}\makebox(0,0)[lb]{\smash{$+A^{-1}$}}}%
    \put(0.09586279,0.06504689){\color[rgb]{0,0,0}\makebox(0,0)[lb]{\smash{$L \ \bigcup$ }}}%
    \put(0.45219313,0.06502602){\color[rgb]{0,0,0}\makebox(0,0)[lb]{\smash{$=(-A^2-A^{-2}) L$}}}%
    \put(-0.00410004,0.23883738){\color[rgb]{0,0,0}\makebox(0,0)[lb]{\smash{K1:}}}%
    \put(-0.0031004,0.06612645){\color[rgb]{0,0,0}\makebox(0,0)[lb]{\smash{K2:}}}%
  \end{picture}%
\endgroup%

%% file: Reidemeister.pdf_tex
\begingroup%
  \makeatletter%
  \providecommand\color[2][]{%
    \errmessage{(Inkscape) Color is used for the text in Inkscape, but the package 'color.sty' is not loaded}%
    \renewcommand\color[2][]{}%
  }%
  \providecommand\transparent[1]{%
    \errmessage{(Inkscape) Transparency is used (non-zero) for the text in Inkscape, but the package 'transparent.sty' is not loaded}%
    \renewcommand\transparent[1]{}%
  }%
  \providecommand\rotatebox[2]{#2}%
  \ifx\svgwidth\undefined%
    \setlength{\unitlength}{221.64836426bp}%
    \ifx\svgscale\undefined%
      \relax%
    \else%
      \setlength{\unitlength}{\unitlength * \real{\svgscale}}%
    \fi%
  \else%
    \setlength{\unitlength}{\svgwidth}%
  \fi%
  \global\let\svgwidth\undefined%
  \global\let\svgscale\undefined%
  \makeatother%
  \begin{picture}(1,0.1979595)%
    \put(0,0){\includegraphics[width=\unitlength]{Reidemeister.pdf}}%
    \put(0.1511898,0.08198808){\color[rgb]{0,0,0}\makebox(0,0)[lb]{\smash{$\sim$}}}%
    \put(0.23048952,0.08198808){\color[rgb]{0,0,0}\makebox(0,0)[lb]{\smash{$\sim$}}}%
    \put(-0.00102922,0.0792536){\color[rgb]{0,0,0}\makebox(0,0)[lb]{\smash{$(R_4)$}}}%
    \put(0.43557502,0.0792536){\color[rgb]{0,0,0}\makebox(0,0)[lb]{\smash{$(R_5)$}}}%
    \put(0.74001303,0.07469615){\color[rgb]{0,0,0}\makebox(0,0)[lb]{\smash{$\sim$}}}%
  \end{picture}%
\endgroup%

%% file: Trivial_curve_rel.pdf_tex
\begingroup%
  \makeatletter%
  \providecommand\color[2][]{%
    \errmessage{(Inkscape) Color is used for the text in Inkscape, but the package 'color.sty' is not loaded}%
    \renewcommand\color[2][]{}%
  }%
  \providecommand\transparent[1]{%
    \errmessage{(Inkscape) Transparency is used (non-zero) for the text in Inkscape, but the package 'transparent.sty' is not loaded}%
    \renewcommand\transparent[1]{}%
  }%
  \providecommand\rotatebox[2]{#2}%
  \ifx\svgwidth\undefined%
    \setlength{\unitlength}{273.41396484bp}%
    \ifx\svgscale\undefined%
      \relax%
    \else%
      \setlength{\unitlength}{\unitlength * \real{\svgscale}}%
    \fi%
  \else%
    \setlength{\unitlength}{\svgwidth}%
  \fi%
  \global\let\svgwidth\undefined%
  \global\let\svgscale\undefined%
  \makeatother%
  \begin{picture}(1,0.39787768)%
    \put(0,0){\includegraphics[width=\unitlength]{Trivial_curve_rel.pdf}}%
    \put(0.19202325,0.2996165){\color[rgb]{0,0,0}\makebox(0,0)[lb]{\smash{$n$}}}%
    \put(0.25187561,0.27079866){\color[rgb]{0,0,0}\makebox(0,0)[lb]{\smash{$=$}}}%
    \put(0.44975041,0.2985081){\color[rgb]{0,0,0}\makebox(0,0)[lb]{\smash{$n-1$}}}%
    \put(-0.00083436,0.07794108){\color[rgb]{0,0,0}\makebox(0,0)[lb]{\smash{$A$}}}%
    \put(0.0805734,0.10772953){\color[rgb]{0,0,0}\makebox(0,0)[lb]{\smash{$n$}}}%
    \put(0.12444486,0.07541818){\color[rgb]{0,0,0}\makebox(0,0)[lb]{\smash{$+A^{-1}$}}}%
    \put(0.37527493,0.09715291){\color[rgb]{0,0,0}\makebox(0,0)[lb]{\smash{$n$}}}%
    \put(0.44177753,0.07424647){\color[rgb]{0,0,0}\makebox(0,0)[lb]{\smash{$=$}}}%
    \put(0.49295855,0.07406182){\color[rgb]{0,0,0}\makebox(0,0)[lb]{\smash{$A^{-1}$}}}%
    \put(0.62489704,0.11276979){\color[rgb]{0,0,0}\makebox(0,0)[lb]{\smash{$n-2$}}}%
    \put(0.74609761,0.08045844){\color[rgb]{0,0,0}\makebox(0,0)[lb]{\smash{$+A$}}}%
    \put(0.94117208,0.09745356){\color[rgb]{0,0,0}\makebox(0,0)[lb]{\smash{$n-1$}}}%
  \end{picture}%
\endgroup%

%% file: Sphere_relation.pdf_tex
\begingroup%
  \makeatletter%
  \providecommand\color[2][]{%
    \errmessage{(Inkscape) Color is used for the text in Inkscape, but the package 'color.sty' is not loaded}%
    \renewcommand\color[2][]{}%
  }%
  \providecommand\transparent[1]{%
    \errmessage{(Inkscape) Transparency is used (non-zero) for the text in Inkscape, but the package 'transparent.sty' is not loaded}%
    \renewcommand\transparent[1]{}%
  }%
  \providecommand\rotatebox[2]{#2}%
  \ifx\svgwidth\undefined%
    \setlength{\unitlength}{276.88706055bp}%
    \ifx\svgscale\undefined%
      \relax%
    \else%
      \setlength{\unitlength}{\unitlength * \real{\svgscale}}%
    \fi%
  \else%
    \setlength{\unitlength}{\svgwidth}%
  \fi%
  \global\let\svgwidth\undefined%
  \global\let\svgscale\undefined%
  \makeatother%
  \begin{picture}(1,0.39647297)%
    \put(0,0){\includegraphics[width=\unitlength]{Sphere_relation.pdf}}%
    \put(-0.00103553,0.1536845){\color[rgb]{0,0,0}\makebox(0,0)[lb]{\smash{$\lbrace 2n+2\rbrace$}}}%
    \put(0.48742132,0.14940954){\color[rgb]{0,0,0}\makebox(0,0)[lb]{\smash{$+\lbrace 2n\rbrace$}}}%
    \put(0.86724068,0.15087898){\color[rgb]{0,0,0}\makebox(0,0)[lb]{\smash{$\equiv 0$}}}%
    \put(0.26776135,0.36269582){\color[rgb]{1,0,0}\makebox(0,0)[lb]{\smash{$n+1$ curves}}}%
    \put(0.69388434,0.36893764){\color[rgb]{1,0,0}\makebox(0,0)[lb]{\smash{$n$ curves}}}%
  \end{picture}%
\endgroup%

%% file: Resolution_n_n+1_n+2.pdf_tex
\begingroup%
  \makeatletter%
  \providecommand\color[2][]{%
    \errmessage{(Inkscape) Color is used for the text in Inkscape, but the package 'color.sty' is not loaded}%
    \renewcommand\color[2][]{}%
  }%
  \providecommand\transparent[1]{%
    \errmessage{(Inkscape) Transparency is used (non-zero) for the text in Inkscape, but the package 'transparent.sty' is not loaded}%
    \renewcommand\transparent[1]{}%
  }%
  \providecommand\rotatebox[2]{#2}%
  \ifx\svgwidth\undefined%
    \setlength{\unitlength}{411.53457031bp}%
    \ifx\svgscale\undefined%
      \relax%
    \else%
      \setlength{\unitlength}{\unitlength * \real{\svgscale}}%
    \fi%
  \else%
    \setlength{\unitlength}{\svgwidth}%
  \fi%
  \global\let\svgwidth\undefined%
  \global\let\svgscale\undefined%
  \makeatother%
  \begin{picture}(1,0.28958003)%
    \put(0,0){\includegraphics[width=\unitlength]{Resolution_n_n+1_n+2.pdf}}%
    \put(0.0884617,0.18913949){\color[rgb]{0,0,0}\makebox(0,0)[lb]{\smash{$0$}}}%
    \put(0.15096061,0.19046263){\color[rgb]{0,0,0}\makebox(0,0)[lb]{\smash{$1$}}}%
    \put(0.19154431,0.18153136){\color[rgb]{0,0,0}\makebox(0,0)[lb]{\smash{$n$}}}%
    \put(0.22058071,0.18514185){\color[rgb]{0,0,0}\makebox(0,0)[lb]{\smash{$n+1$}}}%
    \put(0.28550971,0.10595611){\color[rgb]{0,0,0}\makebox(0,0)[lb]{\smash{$n+2$}}}%
    \put(0.1911971,0.10706685){\color[rgb]{0,0,0}\makebox(0,0)[lb]{\smash{$n+3$}}}%
    \put(0.12693696,0.07128395){\color[rgb]{0,0,0}\makebox(0,0)[lb]{\smash{$2n+2$}}}%
    \put(0.09874291,0.10346475){\color[rgb]{0,0,0}\makebox(0,0)[lb]{\smash{$2n+3$}}}%
    \put(0.62619352,0.1825387){\color[rgb]{0,0,0}\makebox(0,0)[lb]{\smash{$0$}}}%
    \put(0.69802926,0.18653441){\color[rgb]{0,0,0}\makebox(0,0)[lb]{\smash{$1$}}}%
    \put(0.73302616,0.18320467){\color[rgb]{0,0,0}\makebox(0,0)[lb]{\smash{$n$}}}%
    \put(0.81595634,0.18406885){\color[rgb]{0,0,0}\makebox(0,0)[lb]{\smash{$n+1$}}}%
    \put(0.86666138,0.09200521){\color[rgb]{0,0,0}\makebox(0,0)[lb]{\smash{$n+2$}}}%
    \put(0.74656202,0.11223124){\color[rgb]{0,0,0}\makebox(0,0)[lb]{\smash{$n+3$}}}%
    \put(0.63214675,0.11344987){\color[rgb]{0,0,0}\makebox(0,0)[lb]{\smash{$2n+3$}}}%
  \end{picture}%
\endgroup%

%% file: Sphere_relation2.pdf_tex
\begingroup%
  \makeatletter%
  \providecommand\color[2][]{%
    \errmessage{(Inkscape) Color is used for the text in Inkscape, but the package 'color.sty' is not loaded}%
    \renewcommand\color[2][]{}%
  }%
  \providecommand\transparent[1]{%
    \errmessage{(Inkscape) Transparency is used (non-zero) for the text in Inkscape, but the package 'transparent.sty' is not loaded}%
    \renewcommand\transparent[1]{}%
  }%
  \providecommand\rotatebox[2]{#2}%
  \ifx\svgwidth\undefined%
    \setlength{\unitlength}{376.37714844bp}%
    \ifx\svgscale\undefined%
      \relax%
    \else%
      \setlength{\unitlength}{\unitlength * \real{\svgscale}}%
    \fi%
  \else%
    \setlength{\unitlength}{\svgwidth}%
  \fi%
  \global\let\svgwidth\undefined%
  \global\let\svgscale\undefined%
  \makeatother%
  \begin{picture}(1,0.56491584)%
    \put(0,0){\includegraphics[width=\unitlength]{Sphere_relation2.pdf}}%
    \put(-0.0006061,0.41352339){\color[rgb]{0,0,0}\makebox(0,0)[lb]{\smash{$\lbrace 4 \rbrace$}}}%
    \put(0.2699294,0.41567048){\color[rgb]{0,0,0}\makebox(0,0)[lb]{\smash{$+\lbrace 2 \rbrace$}}}%
    \put(0.593069,0.41244982){\color[rgb]{0,0,0}\makebox(0,0)[lb]{\smash{$+\lbrace 2 \rbrace$}}}%
    \put(0.04811704,0.11096515){\color[rgb]{0,0,0}\makebox(0,0)[lb]{\smash{$=\lbrace 4 \rbrace$}}}%
    \put(0.34441783,0.10667096){\color[rgb]{0,0,0}\makebox(0,0)[lb]{\smash{$+\lbrace 2 \rbrace$}}}%
    \put(0.66755749,0.1034503){\color[rgb]{0,0,0}\makebox(0,0)[lb]{\smash{$+\lbrace 2 \rbrace$}}}%
  \end{picture}%
\endgroup%

%% file: Resolution_0_1_2.pdf_tex
\begingroup%
  \makeatletter%
  \providecommand\color[2][]{%
    \errmessage{(Inkscape) Color is used for the text in Inkscape, but the package 'color.sty' is not loaded}%
    \renewcommand\color[2][]{}%
  }%
  \providecommand\transparent[1]{%
    \errmessage{(Inkscape) Transparency is used (non-zero) for the text in Inkscape, but the package 'transparent.sty' is not loaded}%
    \renewcommand\transparent[1]{}%
  }%
  \providecommand\rotatebox[2]{#2}%
  \ifx\svgwidth\undefined%
    \setlength{\unitlength}{128.97402344bp}%
    \ifx\svgscale\undefined%
      \relax%
    \else%
      \setlength{\unitlength}{\unitlength * \real{\svgscale}}%
    \fi%
  \else%
    \setlength{\unitlength}{\svgwidth}%
  \fi%
  \global\let\svgwidth\undefined%
  \global\let\svgscale\undefined%
  \makeatother%
  \begin{picture}(1,0.52148837)%
    \put(0,0){\includegraphics[width=\unitlength]{Resolution_0_1_2.pdf}}%
    \put(0.06104652,0.46135779){\color[rgb]{0,0,0}\makebox(0,0)[lb]{\smash{$1$}}}%
    \put(0.19629149,0.25566535){\color[rgb]{0,0,0}\makebox(0,0)[lb]{\smash{$2$}}}%
    \put(0.54745393,0.26582303){\color[rgb]{0,0,0}\makebox(0,0)[lb]{\smash{$3$}}}%
    \put(0.75388044,0.26836244){\color[rgb]{0,0,0}\makebox(0,0)[lb]{\smash{$4$}}}%
    \put(0.46440878,0.46643663){\color[rgb]{1,0,0}\makebox(0,0)[lb]{\smash{$\gamma$}}}%
    \put(0.40746347,0.01696053){\color[rgb]{0,0,1}\makebox(0,0)[lb]{\smash{$\delta$}}}%
  \end{picture}%
\endgroup%

%% file: Resolution0_1_2bis.pdf_tex
\begingroup%
  \makeatletter%
  \providecommand\color[2][]{%
    \errmessage{(Inkscape) Color is used for the text in Inkscape, but the package 'color.sty' is not loaded}%
    \renewcommand\color[2][]{}%
  }%
  \providecommand\transparent[1]{%
    \errmessage{(Inkscape) Transparency is used (non-zero) for the text in Inkscape, but the package 'transparent.sty' is not loaded}%
    \renewcommand\transparent[1]{}%
  }%
  \providecommand\rotatebox[2]{#2}%
  \ifx\svgwidth\undefined%
    \setlength{\unitlength}{385.53190918bp}%
    \ifx\svgscale\undefined%
      \relax%
    \else%
      \setlength{\unitlength}{\unitlength * \real{\svgscale}}%
    \fi%
  \else%
    \setlength{\unitlength}{\svgwidth}%
  \fi%
  \global\let\svgwidth\undefined%
  \global\let\svgscale\undefined%
  \makeatother%
  \begin{picture}(1,0.47260668)%
    \put(0,0){\includegraphics[width=\unitlength]{Resolution0_1_2bis.pdf}}%
    \put(0.00210741,0.46138209){\color[rgb]{0,0,0}\makebox(0,0)[lb]{\smash{$++++$}}}%
    \put(0.34615987,0.46215774){\color[rgb]{0,0,0}\makebox(0,0)[lb]{\smash{$-+++$}}}%
    \put(0.68375121,0.46215774){\color[rgb]{0,0,0}\makebox(0,0)[lb]{\smash{$++-+$}}}%
    \put(0.00611704,0.21463669){\color[rgb]{0,0,0}\makebox(0,0)[lb]{\smash{$----$}}}%
    \put(0.36550766,0.21302551){\color[rgb]{0,0,0}\makebox(0,0)[lb]{\smash{$---+$}}}%
    \put(0.7220259,0.21779921){\color[rgb]{0,0,0}\makebox(0,0)[lb]{\smash{$-+--$}}}%
  \end{picture}%
\endgroup%

%% file: Resolution_Tore.pdf_tex
\begingroup%
  \makeatletter%
  \providecommand\color[2][]{%
    \errmessage{(Inkscape) Color is used for the text in Inkscape, but the package 'color.sty' is not loaded}%
    \renewcommand\color[2][]{}%
  }%
  \providecommand\transparent[1]{%
    \errmessage{(Inkscape) Transparency is used (non-zero) for the text in Inkscape, but the package 'transparent.sty' is not loaded}%
    \renewcommand\transparent[1]{}%
  }%
  \providecommand\rotatebox[2]{#2}%
  \ifx\svgwidth\undefined%
    \setlength{\unitlength}{59.7208374bp}%
    \ifx\svgscale\undefined%
      \relax%
    \else%
      \setlength{\unitlength}{\unitlength * \real{\svgscale}}%
    \fi%
  \else%
    \setlength{\unitlength}{\svgwidth}%
  \fi%
  \global\let\svgwidth\undefined%
  \global\let\svgscale\undefined%
  \makeatother%
  \begin{picture}(1,1.4723783)%
    \put(0,0){\includegraphics[width=\unitlength]{Resolution_Tore.pdf}}%
    \put(0.32575843,1.44166932){\color[rgb]{1,0,0}\makebox(0,0)[lb]{\smash{$n$ strands}}}%
    \put(0.16752053,0.88363735){\color[rgb]{0,0,0}\makebox(0,0)[lb]{\smash{$0$}}}%
    \put(0.55695659,0.8871448){\color[rgb]{0,0,0}\makebox(0,0)[lb]{\smash{$n-1$}}}%
  \end{picture}%
\endgroup%

%% file: Rel_Tore_Deux_Trous1.pdf_tex
\begingroup%
  \makeatletter%
  \providecommand\color[2][]{%
    \errmessage{(Inkscape) Color is used for the text in Inkscape, but the package 'color.sty' is not loaded}%
    \renewcommand\color[2][]{}%
  }%
  \providecommand\transparent[1]{%
    \errmessage{(Inkscape) Transparency is used (non-zero) for the text in Inkscape, but the package 'transparent.sty' is not loaded}%
    \renewcommand\transparent[1]{}%
  }%
  \providecommand\rotatebox[2]{#2}%
  \ifx\svgwidth\undefined%
    \setlength{\unitlength}{426.1296875bp}%
    \ifx\svgscale\undefined%
      \relax%
    \else%
      \setlength{\unitlength}{\unitlength * \real{\svgscale}}%
    \fi%
  \else%
    \setlength{\unitlength}{\svgwidth}%
  \fi%
  \global\let\svgwidth\undefined%
  \global\let\svgscale\undefined%
  \makeatother%
  \begin{picture}(1,0.15641466)%
    \put(0,0){\includegraphics[width=\unitlength]{Rel_Tore_Deux_Trous1.pdf}}%
    \put(0.10204911,0.10354605){\color[rgb]{0,0,1}\makebox(0,0)[lb]{\smash{$a$}}}%
    \put(0.10837054,0.04481038){\color[rgb]{0,0,1}\makebox(0,0)[lb]{\smash{$b$}}}%
    \put(0.35289732,0.14467381){\color[rgb]{0,0,1}\makebox(0,0)[lb]{\smash{$a+1$}}}%
    \put(0.36633037,0.00154736){\color[rgb]{0,0,1}\makebox(0,0)[lb]{\smash{$b+1$}}}%
    \put(0.55202734,0.13945775){\color[rgb]{0,0,1}\makebox(0,0)[lb]{\smash{$a+b$}}}%
    \put(0.79739976,0.14102731){\color[rgb]{0,0,1}\makebox(0,0)[lb]{\smash{$a+b+2$}}}%
    \put(-0.00053534,0.05929394){\color[rgb]{0,0,0}\makebox(0,0)[lb]{\smash{$A^{-1}$}}}%
    \put(0.25290862,0.05661199){\color[rgb]{0,0,0}\makebox(0,0)[lb]{\smash{$-A$}}}%
    \put(0.48891993,0.05929394){\color[rgb]{0,0,0}\makebox(0,0)[lb]{\smash{$=-A$}}}%
    \put(0.75175071,0.0606349){\color[rgb]{0,0,0}\makebox(0,0)[lb]{\smash{$+A^{-1}$}}}%
  \end{picture}%
\endgroup%

%% file: Rel_Tore_Deux_Trous2.pdf_tex
\begingroup%
  \makeatletter%
  \providecommand\color[2][]{%
    \errmessage{(Inkscape) Color is used for the text in Inkscape, but the package 'color.sty' is not loaded}%
    \renewcommand\color[2][]{}%
  }%
  \providecommand\transparent[1]{%
    \errmessage{(Inkscape) Transparency is used (non-zero) for the text in Inkscape, but the package 'transparent.sty' is not loaded}%
    \renewcommand\transparent[1]{}%
  }%
  \providecommand\rotatebox[2]{#2}%
  \ifx\svgwidth\undefined%
    \setlength{\unitlength}{435.19438477bp}%
    \ifx\svgscale\undefined%
      \relax%
    \else%
      \setlength{\unitlength}{\unitlength * \real{\svgscale}}%
    \fi%
  \else%
    \setlength{\unitlength}{\svgwidth}%
  \fi%
  \global\let\svgwidth\undefined%
  \global\let\svgscale\undefined%
  \makeatother%
  \begin{picture}(1,0.15315671)%
    \put(0,0){\includegraphics[width=\unitlength]{Rel_Tore_Deux_Trous2.pdf}}%
    \put(0.0802279,0.10007626){\color[rgb]{0,0,1}\makebox(0,0)[lb]{\smash{$a$}}}%
    \put(0.08641766,0.04256399){\color[rgb]{0,0,1}\makebox(0,0)[lb]{\smash{$b$}}}%
    \put(0.35211203,0.14166041){\color[rgb]{0,0,1}\makebox(0,0)[lb]{\smash{$a+1$}}}%
    \put(0.36526528,0.00151513){\color[rgb]{0,0,1}\makebox(0,0)[lb]{\smash{$b+1$}}}%
    \put(0.70203326,0.1299878){\color[rgb]{0,0,1}\makebox(0,0)[lb]{\smash{$a+b+2$}}}%
    \put(0.95936444,0.13546377){\color[rgb]{0,0,1}\makebox(0,0)[lb]{\smash{$a+b$}}}%
    \put(-0.00052419,0.05674587){\color[rgb]{0,0,0}\makebox(0,0)[lb]{\smash{$A$}}}%
    \put(0.22269297,0.05674587){\color[rgb]{0,0,0}\makebox(0,0)[lb]{\smash{$-A^{-1}$}}}%
    \put(0.48530135,0.0580589){\color[rgb]{0,0,0}\makebox(0,0)[lb]{\smash{$=A$}}}%
    \put(0.74265761,0.05937193){\color[rgb]{0,0,0}\makebox(0,0)[lb]{\smash{$-A^{-1}$}}}%
  \end{picture}%
\endgroup%

%% file: Resol_Tore_Deux_Trous1.pdf_tex
\begingroup%
  \makeatletter%
  \providecommand\color[2][]{%
    \errmessage{(Inkscape) Color is used for the text in Inkscape, but the package 'color.sty' is not loaded}%
    \renewcommand\color[2][]{}%
  }%
  \providecommand\transparent[1]{%
    \errmessage{(Inkscape) Transparency is used (non-zero) for the text in Inkscape, but the package 'transparent.sty' is not loaded}%
    \renewcommand\transparent[1]{}%
  }%
  \providecommand\rotatebox[2]{#2}%
  \ifx\svgwidth\undefined%
    \setlength{\unitlength}{226.06918945bp}%
    \ifx\svgscale\undefined%
      \relax%
    \else%
      \setlength{\unitlength}{\unitlength * \real{\svgscale}}%
    \fi%
  \else%
    \setlength{\unitlength}{\svgwidth}%
  \fi%
  \global\let\svgwidth\undefined%
  \global\let\svgscale\undefined%
  \makeatother%
  \begin{picture}(1,0.4094522)%
    \put(0,0){\includegraphics[width=\unitlength]{Resol_Tore_Deux_Trous1.pdf}}%
    \put(0.14131409,0.37968677){\color[rgb]{0,0,1}\makebox(0,0)[lb]{\smash{$a$}}}%
    \put(0.17901342,0.00468252){\color[rgb]{0,0,1}\makebox(0,0)[lb]{\smash{$b$}}}%
    \put(0.48552749,0.18159722){\color[rgb]{0,0,0}\makebox(0,0)[lb]{\smash{$=$}}}%
    \put(0.70384416,0.37892711){\color[rgb]{0,0,1}\makebox(0,0)[lb]{\smash{$a+1$}}}%
    \put(0.74154349,0.00392286){\color[rgb]{0,0,1}\makebox(0,0)[lb]{\smash{$b+1$}}}%
  \end{picture}%
\endgroup%

%% file: Resol_Tore_Deux_Trous2.pdf_tex
\begingroup%
  \makeatletter%
  \providecommand\color[2][]{%
    \errmessage{(Inkscape) Color is used for the text in Inkscape, but the package 'color.sty' is not loaded}%
    \renewcommand\color[2][]{}%
  }%
  \providecommand\transparent[1]{%
    \errmessage{(Inkscape) Transparency is used (non-zero) for the text in Inkscape, but the package 'transparent.sty' is not loaded}%
    \renewcommand\transparent[1]{}%
  }%
  \providecommand\rotatebox[2]{#2}%
  \ifx\svgwidth\undefined%
    \setlength{\unitlength}{226.06918945bp}%
    \ifx\svgscale\undefined%
      \relax%
    \else%
      \setlength{\unitlength}{\unitlength * \real{\svgscale}}%
    \fi%
  \else%
    \setlength{\unitlength}{\svgwidth}%
  \fi%
  \global\let\svgwidth\undefined%
  \global\let\svgscale\undefined%
  \makeatother%
  \begin{picture}(1,0.40945225)%
    \put(0,0){\includegraphics[width=\unitlength]{Resol_Tore_Deux_Trous2.pdf}}%
    \put(0.14131408,0.37968684){\color[rgb]{0,0,1}\makebox(0,0)[lb]{\smash{$a$}}}%
    \put(0.17901341,0.00468265){\color[rgb]{0,0,1}\makebox(0,0)[lb]{\smash{$b$}}}%
    \put(0.48552754,0.1815973){\color[rgb]{0,0,0}\makebox(0,0)[lb]{\smash{$=$}}}%
    \put(0.70384414,0.37892719){\color[rgb]{0,0,1}\makebox(0,0)[lb]{\smash{$a+1$}}}%
    \put(0.74154348,0.00392288){\color[rgb]{0,0,1}\makebox(0,0)[lb]{\smash{$b+1$}}}%
  \end{picture}%
\endgroup%

%% file: Resol_Sigma_g_2.pdf_tex
\begingroup%
  \makeatletter%
  \providecommand\color[2][]{%
    \errmessage{(Inkscape) Color is used for the text in Inkscape, but the package 'color.sty' is not loaded}%
    \renewcommand\color[2][]{}%
  }%
  \providecommand\transparent[1]{%
    \errmessage{(Inkscape) Transparency is used (non-zero) for the text in Inkscape, but the package 'transparent.sty' is not loaded}%
    \renewcommand\transparent[1]{}%
  }%
  \providecommand\rotatebox[2]{#2}%
  \ifx\svgwidth\undefined%
    \setlength{\unitlength}{122.0293335bp}%
    \ifx\svgscale\undefined%
      \relax%
    \else%
      \setlength{\unitlength}{\unitlength * \real{\svgscale}}%
    \fi%
  \else%
    \setlength{\unitlength}{\svgwidth}%
  \fi%
  \global\let\svgwidth\undefined%
  \global\let\svgscale\undefined%
  \makeatother%
  \begin{picture}(1,0.93792396)%
    \put(0,0){\includegraphics[width=\unitlength]{Resol_Sigma_g_2.pdf}}%
    \put(0.61914762,0.84528641){\color[rgb]{1,0,0}\makebox(0,0)[lb]{\smash{$\gamma$}}}%
    \put(0.66340641,0.57275628){\color[rgb]{0,0,1}\makebox(0,0)[lb]{\smash{$\delta$}}}%
  \end{picture}%
\endgroup%

%% file: Resol_Sigma_g_2bis.pdf_tex
\begingroup%
  \makeatletter%
  \providecommand\color[2][]{%
    \errmessage{(Inkscape) Color is used for the text in Inkscape, but the package 'color.sty' is not loaded}%
    \renewcommand\color[2][]{}%
  }%
  \providecommand\transparent[1]{%
    \errmessage{(Inkscape) Transparency is used (non-zero) for the text in Inkscape, but the package 'transparent.sty' is not loaded}%
    \renewcommand\transparent[1]{}%
  }%
  \providecommand\rotatebox[2]{#2}%
  \ifx\svgwidth\undefined%
    \setlength{\unitlength}{396.06928711bp}%
    \ifx\svgscale\undefined%
      \relax%
    \else%
      \setlength{\unitlength}{\unitlength * \real{\svgscale}}%
    \fi%
  \else%
    \setlength{\unitlength}{\svgwidth}%
  \fi%
  \global\let\svgwidth\undefined%
  \global\let\svgscale\undefined%
  \makeatother%
  \begin{picture}(1,0.52465333)%
    \put(0,0){\includegraphics[width=\unitlength]{Resol_Sigma_g_2bis.pdf}}%
    \put(0.00468666,0.50539024){\color[rgb]{0,0,0}\makebox(0,0)[lb]{\smash{$+++$}}}%
    \put(0.27246779,0.50651849){\color[rgb]{0,0,0}\makebox(0,0)[lb]{\smash{$++-$}}}%
    \put(0.53312319,0.50702857){\color[rgb]{0,0,0}\makebox(0,0)[lb]{\smash{$+-+$}}}%
    \put(0.80552611,0.51202139){\color[rgb]{0,0,0}\makebox(0,0)[lb]{\smash{$-++$}}}%
    \put(0.00468667,0.22739184){\color[rgb]{0,0,0}\makebox(0,0)[lb]{\smash{$---$}}}%
    \put(0.2760539,0.22943219){\color[rgb]{0,0,0}\makebox(0,0)[lb]{\smash{$+--$}}}%
    \put(0.53619919,0.22739184){\color[rgb]{0,0,0}\makebox(0,0)[lb]{\smash{$-+-$}}}%
    \put(0.79634451,0.22637165){\color[rgb]{0,0,0}\makebox(0,0)[lb]{\smash{$--+$}}}%
    \put(0.01284809,0.29166304){\color[rgb]{0,0,0}\makebox(0,0)[lb]{\smash{$deg=3$}}}%
    \put(0.28319515,0.29268319){\color[rgb]{0,0,0}\makebox(0,0)[lb]{\smash{$deg=5$}}}%
    \put(0.54844132,0.29064285){\color[rgb]{0,0,0}\makebox(0,0)[lb]{\smash{$deg=1$}}}%
    \put(0.80756642,0.28860248){\color[rgb]{0,0,0}\makebox(0,0)[lb]{\smash{$deg=1$}}}%
    \put(0.0107052,0.0034006){\color[rgb]{0,0,0}\makebox(0,0)[lb]{\smash{$deg=1$}}}%
    \put(0.28915577,0.0048433){\color[rgb]{0,0,0}\makebox(0,0)[lb]{\smash{$deg=3$}}}%
    \put(0.55606432,0.00628607){\color[rgb]{0,0,0}\makebox(0,0)[lb]{\smash{$deg=3$}}}%
    \put(0.8085454,0.0048433){\color[rgb]{0,0,0}\makebox(0,0)[lb]{\smash{$deg=1$}}}%
  \end{picture}%
\endgroup%

%% file: Saucisson.pdf_tex
\begingroup%
  \makeatletter%
  \providecommand\color[2][]{%
    \errmessage{(Inkscape) Color is used for the text in Inkscape, but the package 'color.sty' is not loaded}%
    \renewcommand\color[2][]{}%
  }%
  \providecommand\transparent[1]{%
    \errmessage{(Inkscape) Transparency is used (non-zero) for the text in Inkscape, but the package 'transparent.sty' is not loaded}%
    \renewcommand\transparent[1]{}%
  }%
  \providecommand\rotatebox[2]{#2}%
  \ifx\svgwidth\undefined%
    \setlength{\unitlength}{260.73310547bp}%
    \ifx\svgscale\undefined%
      \relax%
    \else%
      \setlength{\unitlength}{\unitlength * \real{\svgscale}}%
    \fi%
  \else%
    \setlength{\unitlength}{\svgwidth}%
  \fi%
  \global\let\svgwidth\undefined%
  \global\let\svgscale\undefined%
  \makeatother%
  \begin{picture}(1,0.22989008)%
    \put(0,0){\includegraphics[width=\unitlength]{Saucisson.pdf}}%
  \end{picture}%
\endgroup%

%% file: Diagram_D_ab_k.pdf_tex
\begingroup%
  \makeatletter%
  \providecommand\color[2][]{%
    \errmessage{(Inkscape) Color is used for the text in Inkscape, but the package 'color.sty' is not loaded}%
    \renewcommand\color[2][]{}%
  }%
  \providecommand\transparent[1]{%
    \errmessage{(Inkscape) Transparency is used (non-zero) for the text in Inkscape, but the package 'transparent.sty' is not loaded}%
    \renewcommand\transparent[1]{}%
  }%
  \providecommand\rotatebox[2]{#2}%
  \ifx\svgwidth\undefined%
    \setlength{\unitlength}{316.10256348bp}%
    \ifx\svgscale\undefined%
      \relax%
    \else%
      \setlength{\unitlength}{\unitlength * \real{\svgscale}}%
    \fi%
  \else%
    \setlength{\unitlength}{\svgwidth}%
  \fi%
  \global\let\svgwidth\undefined%
  \global\let\svgscale\undefined%
  \makeatother%
  \begin{picture}(1,0.30067205)%
    \put(0,0){\includegraphics[width=\unitlength]{Diagram_D_ab_k.pdf}}%
    \put(0.23208987,0.15368277){\color[rgb]{1,0,0}\makebox(0,0)[lb]{\smash{$a$}}}%
    \put(0.23656379,0.08018286){\color[rgb]{1,0,0}\makebox(0,0)[lb]{\smash{$b$}}}%
    \put(0.1221596,0.00348732){\color[rgb]{0,0,1}\makebox(0,0)[lb]{\smash{$m \ \textrm{curves}$}}}%
    \put(0.07638172,0.28484452){\color[rgb]{0,0,0}\makebox(0,0)[lb]{\smash{$k \ \textrm{pairs of pants}$}}}%
    \put(-0.00072168,0.22197808){\color[rgb]{0,0,0}\makebox(0,0)[lb]{\smash{$k_0 \ \textrm{pairs of pants}$}}}%
    \put(0.90827754,0.03839432){\color[rgb]{1,0,0}\makebox(0,0)[lb]{\smash{$a+b$}}}%
    \put(0.66457519,0.00208595){\color[rgb]{0,0,1}\makebox(0,0)[lb]{\smash{$m \ \textrm{curves}$}}}%
    \put(0.66670215,0.28073156){\color[rgb]{0,0,0}\makebox(0,0)[lb]{\smash{$k \ \textrm{pairs of pants}$}}}%
    \put(0.5416939,0.22057672){\color[rgb]{0,0,0}\makebox(0,0)[lb]{\smash{$k_0 \ \textrm{pairs of pants}$}}}%
    \put(0.51631765,0.111673){\color[rgb]{0,0,0}\makebox(0,0)[lb]{\smash{or}}}%
  \end{picture}%
\endgroup%

%% file: Diagram_D_ab_k2.pdf_tex
\begingroup%
  \makeatletter%
  \providecommand\color[2][]{%
    \errmessage{(Inkscape) Color is used for the text in Inkscape, but the package 'color.sty' is not loaded}%
    \renewcommand\color[2][]{}%
  }%
  \providecommand\transparent[1]{%
    \errmessage{(Inkscape) Transparency is used (non-zero) for the text in Inkscape, but the package 'transparent.sty' is not loaded}%
    \renewcommand\transparent[1]{}%
  }%
  \providecommand\rotatebox[2]{#2}%
  \ifx\svgwidth\undefined%
    \setlength{\unitlength}{152.81339111bp}%
    \ifx\svgscale\undefined%
      \relax%
    \else%
      \setlength{\unitlength}{\unitlength * \real{\svgscale}}%
    \fi%
  \else%
    \setlength{\unitlength}{\svgwidth}%
  \fi%
  \global\let\svgwidth\undefined%
  \global\let\svgscale\undefined%
  \makeatother%
  \begin{picture}(1,0.37707307)%
    \put(0,0){\includegraphics[width=\unitlength]{Diagram_D_ab_k2.pdf}}%
    \put(-0.00149284,0.1770863){\color[rgb]{1,0,0}\makebox(0,0)[lb]{\smash{$a+b$}}}%
    \put(0.19575988,0.00881386){\color[rgb]{0,0,0}\makebox(0,0)[lb]{\smash{$_lD_{a,b}^{k_0}$}}}%
    \put(0.839776,0.25140994){\color[rgb]{1,0,0}\makebox(0,0)[lb]{\smash{$a+b$}}}%
    \put(0.66308992,0.01208909){\color[rgb]{0,0,0}\makebox(0,0)[lb]{\smash{$_rD_{a,b}^{k_0}$}}}%
  \end{picture}%
\endgroup%

%% file: Lemma_LeftRight1.pdf_tex
\begingroup%
  \makeatletter%
  \providecommand\color[2][]{%
    \errmessage{(Inkscape) Color is used for the text in Inkscape, but the package 'color.sty' is not loaded}%
    \renewcommand\color[2][]{}%
  }%
  \providecommand\transparent[1]{%
    \errmessage{(Inkscape) Transparency is used (non-zero) for the text in Inkscape, but the package 'transparent.sty' is not loaded}%
    \renewcommand\transparent[1]{}%
  }%
  \providecommand\rotatebox[2]{#2}%
  \ifx\svgwidth\undefined%
    \setlength{\unitlength}{194.58708496bp}%
    \ifx\svgscale\undefined%
      \relax%
    \else%
      \setlength{\unitlength}{\unitlength * \real{\svgscale}}%
    \fi%
  \else%
    \setlength{\unitlength}{\svgwidth}%
  \fi%
  \global\let\svgwidth\undefined%
  \global\let\svgscale\undefined%
  \makeatother%
  \begin{picture}(1,0.28988385)%
    \put(0,0){\includegraphics[width=\unitlength]{Lemma_LeftRight1.pdf}}%
    \put(0.04409258,0.18346266){\color[rgb]{1,0,0}\makebox(0,0)[lb]{\smash{$a$}}}%
    \put(0.27518967,0.17885257){\color[rgb]{1,0,0}\makebox(0,0)[lb]{\smash{$b$}}}%
    \put(0.44152879,0.08204072){\color[rgb]{0,0,0}\makebox(0,0)[lb]{\smash{$=$}}}%
    \put(0.6117179,0.18059595){\color[rgb]{1,0,0}\makebox(0,0)[lb]{\smash{$a$}}}%
    \put(0.84281506,0.17598585){\color[rgb]{1,0,0}\makebox(0,0)[lb]{\smash{$b+1$}}}%
  \end{picture}%
\endgroup%

%% file: Lemma_LeftRight2.pdf_tex
\begingroup%
  \makeatletter%
  \providecommand\color[2][]{%
    \errmessage{(Inkscape) Color is used for the text in Inkscape, but the package 'color.sty' is not loaded}%
    \renewcommand\color[2][]{}%
  }%
  \providecommand\transparent[1]{%
    \errmessage{(Inkscape) Transparency is used (non-zero) for the text in Inkscape, but the package 'transparent.sty' is not loaded}%
    \renewcommand\transparent[1]{}%
  }%
  \providecommand\rotatebox[2]{#2}%
  \ifx\svgwidth\undefined%
    \setlength{\unitlength}{449.53955078bp}%
    \ifx\svgscale\undefined%
      \relax%
    \else%
      \setlength{\unitlength}{\unitlength * \real{\svgscale}}%
    \fi%
  \else%
    \setlength{\unitlength}{\svgwidth}%
  \fi%
  \global\let\svgwidth\undefined%
  \global\let\svgscale\undefined%
  \makeatother%
  \begin{picture}(1,0.12652667)%
    \put(0,0){\includegraphics[width=\unitlength]{Lemma_LeftRight2.pdf}}%
    \put(0.07351491,0.06615928){\color[rgb]{1,0,0}\makebox(0,0)[lb]{\smash{$a+1$}}}%
    \put(0.19691696,0.05787193){\color[rgb]{1,0,0}\makebox(0,0)[lb]{\smash{$b$}}}%
    \put(-0.00050746,0.04870679){\color[rgb]{0,0,0}\makebox(0,0)[lb]{\smash{$A^{-1}$}}}%
    \put(0.26842657,0.04761225){\color[rgb]{0,0,0}\makebox(0,0)[lb]{\smash{$+A$}}}%
    \put(0.41898112,0.07233015){\color[rgb]{1,0,0}\makebox(0,0)[lb]{\smash{$a-b+1$}}}%
    \put(0.60168376,0.0611979){\color[rgb]{1,0,0}\makebox(0,0)[lb]{\smash{$a$}}}%
    \put(0.70093411,0.06562196){\color[rgb]{1,0,0}\makebox(0,0)[lb]{\smash{$b+1$}}}%
    \put(0.50605198,0.05010112){\color[rgb]{0,0,0}\makebox(0,0)[lb]{\smash{$=A$}}}%
    \put(0.76990144,0.05409111){\color[rgb]{0,0,0}\makebox(0,0)[lb]{\smash{$+A^{-1}$}}}%
    \put(0.92554057,0.07372447){\color[rgb]{1,0,0}\makebox(0,0)[lb]{\smash{$a-b+1$}}}%
  \end{picture}%
\endgroup%

%% file: Lemma_LeftRight3.pdf_tex
\begingroup%
  \makeatletter%
  \providecommand\color[2][]{%
    \errmessage{(Inkscape) Color is used for the text in Inkscape, but the package 'color.sty' is not loaded}%
    \renewcommand\color[2][]{}%
  }%
  \providecommand\transparent[1]{%
    \errmessage{(Inkscape) Transparency is used (non-zero) for the text in Inkscape, but the package 'transparent.sty' is not loaded}%
    \renewcommand\transparent[1]{}%
  }%
  \providecommand\rotatebox[2]{#2}%
  \ifx\svgwidth\undefined%
    \setlength{\unitlength}{213.87888184bp}%
    \ifx\svgscale\undefined%
      \relax%
    \else%
      \setlength{\unitlength}{\unitlength * \real{\svgscale}}%
    \fi%
  \else%
    \setlength{\unitlength}{\svgwidth}%
  \fi%
  \global\let\svgwidth\undefined%
  \global\let\svgscale\undefined%
  \makeatother%
  \begin{picture}(1,0.26522751)%
    \put(0,0){\includegraphics[width=\unitlength]{Lemma_LeftRight3.pdf}}%
    \put(0.00489085,0.14127546){\color[rgb]{1,0,0}\makebox(0,0)[lb]{\smash{$a+1$}}}%
    \put(0.26426243,0.12385677){\color[rgb]{1,0,0}\makebox(0,0)[lb]{\smash{$b$}}}%
    \put(0.41456425,0.10229255){\color[rgb]{0,0,0}\makebox(0,0)[lb]{\smash{$=A^2$}}}%
    \put(0.64838448,0.12517977){\color[rgb]{1,0,0}\makebox(0,0)[lb]{\smash{$a$}}}%
    \put(0.8569931,0.1344785){\color[rgb]{1,0,0}\makebox(0,0)[lb]{\smash{$b+1$}}}%
  \end{picture}%
\endgroup%

%% file: Courbes_Saucisson.pdf_tex
\begingroup%
  \makeatletter%
  \providecommand\color[2][]{%
    \errmessage{(Inkscape) Color is used for the text in Inkscape, but the package 'color.sty' is not loaded}%
    \renewcommand\color[2][]{}%
  }%
  \providecommand\transparent[1]{%
    \errmessage{(Inkscape) Transparency is used (non-zero) for the text in Inkscape, but the package 'transparent.sty' is not loaded}%
    \renewcommand\transparent[1]{}%
  }%
  \providecommand\rotatebox[2]{#2}%
  \ifx\svgwidth\undefined%
    \setlength{\unitlength}{443.78110352bp}%
    \ifx\svgscale\undefined%
      \relax%
    \else%
      \setlength{\unitlength}{\unitlength * \real{\svgscale}}%
    \fi%
  \else%
    \setlength{\unitlength}{\svgwidth}%
  \fi%
  \global\let\svgwidth\undefined%
  \global\let\svgscale\undefined%
  \makeatother%
  \begin{picture}(1,0.20265682)%
    \put(0,0){\includegraphics[width=\unitlength]{Courbes_Saucisson.pdf}}%
    \put(-0.00051405,0.14478613){\color[rgb]{0,0,1}\makebox(0,0)[lb]{\smash{$n_1\geqslant 0$}}}%
    \put(0.20579057,0.02527792){\color[rgb]{0,0,1}\makebox(0,0)[lb]{\smash{$m_1\geqslant 0$
   curves}}}%
    \put(0.36545747,0.19138296){\color[rgb]{0,0,1}\makebox(0,0)[lb]{\smash{$m_2\geqslant 0$
   curves}}}%
    \put(0.68479125,0.01819591){\color[rgb]{0,0,1}\makebox(0,0)[lb]{\smash{$m_{g-1}\geqslant 0$
   curves}}}%
    \put(0.8859587,0.10615704){\color[rgb]{0,0,1}\makebox(0,0)[lb]{\smash{$n_2 \geqslant 0$
  curves}}}%
    \put(-0.00002644,0.12543407){\color[rgb]{0,0,1}\makebox(0,0)[lb]{\smash{curves}}}%
  \end{picture}%
\endgroup%

%% file: Diagram_D_ab_0.pdf_tex
\begingroup%
  \makeatletter%
  \providecommand\color[2][]{%
    \errmessage{(Inkscape) Color is used for the text in Inkscape, but the package 'color.sty' is not loaded}%
    \renewcommand\color[2][]{}%
  }%
  \providecommand\transparent[1]{%
    \errmessage{(Inkscape) Transparency is used (non-zero) for the text in Inkscape, but the package 'transparent.sty' is not loaded}%
    \renewcommand\transparent[1]{}%
  }%
  \providecommand\rotatebox[2]{#2}%
  \ifx\svgwidth\undefined%
    \setlength{\unitlength}{358.21242676bp}%
    \ifx\svgscale\undefined%
      \relax%
    \else%
      \setlength{\unitlength}{\unitlength * \real{\svgscale}}%
    \fi%
  \else%
    \setlength{\unitlength}{\svgwidth}%
  \fi%
  \global\let\svgwidth\undefined%
  \global\let\svgscale\undefined%
  \makeatother%
  \begin{picture}(1,0.22873421)%
    \put(0,0){\includegraphics[width=\unitlength]{Diagram_D_ab_0.pdf}}%
    \put(0.01628123,0.02813219){\color[rgb]{0,0,1}\makebox(0,0)[lb]{\smash{$n_1-2\geqslant 0$}}}%
    \put(0.8567655,0.03216317){\color[rgb]{0,0,1}\makebox(0,0)[lb]{\smash{$n_2 \geqslant 0$
}}}%
    \put(0.04000914,0.00020937){\color[rgb]{0,0,1}\makebox(0,0)[lb]{\smash{curves}}}%
    \put(0.13778372,0.17507614){\color[rgb]{1,0,0}\makebox(0,0)[lb]{\smash{$a$}}}%
    \put(0.14342369,0.11416453){\color[rgb]{1,0,0}\makebox(0,0)[lb]{\smash{$b$}}}%
    \put(0.873799,0.00756921){\color[rgb]{0,0,1}\makebox(0,0)[lb]{\smash{curves}}}%
  \end{picture}%
\endgroup%

%% file: Diagram_D_ab_N.pdf_tex
\begingroup%
  \makeatletter%
  \providecommand\color[2][]{%
    \errmessage{(Inkscape) Color is used for the text in Inkscape, but the package 'color.sty' is not loaded}%
    \renewcommand\color[2][]{}%
  }%
  \providecommand\transparent[1]{%
    \errmessage{(Inkscape) Transparency is used (non-zero) for the text in Inkscape, but the package 'transparent.sty' is not loaded}%
    \renewcommand\transparent[1]{}%
  }%
  \providecommand\rotatebox[2]{#2}%
  \ifx\svgwidth\undefined%
    \setlength{\unitlength}{358.21242676bp}%
    \ifx\svgscale\undefined%
      \relax%
    \else%
      \setlength{\unitlength}{\unitlength * \real{\svgscale}}%
    \fi%
  \else%
    \setlength{\unitlength}{\svgwidth}%
  \fi%
  \global\let\svgwidth\undefined%
  \global\let\svgscale\undefined%
  \makeatother%
  \begin{picture}(1,0.22873421)%
    \put(0,0){\includegraphics[width=\unitlength]{Diagram_D_ab_N.pdf}}%
    \put(0.01628123,0.0281322){\color[rgb]{0,0,1}\makebox(0,0)[lb]{\smash{$n_1-2\geqslant 0$}}}%
    \put(0.85676553,0.03216316){\color[rgb]{0,0,1}\makebox(0,0)[lb]{\smash{$n_2 \geqslant 0$
}}}%
    \put(0.04000914,0.00020936){\color[rgb]{0,0,1}\makebox(0,0)[lb]{\smash{curves}}}%
    \put(0.87379903,0.00756921){\color[rgb]{0,0,1}\makebox(0,0)[lb]{\smash{curves}}}%
    \put(0.84268976,0.17353386){\color[rgb]{1,0,0}\makebox(0,0)[lb]{\smash{$a$}}}%
    \put(0.83704978,0.11713422){\color[rgb]{1,0,0}\makebox(0,0)[lb]{\smash{$b$}}}%
  \end{picture}%
\endgroup%

%% file: Trivial_curve_rel2.pdf_tex
\begingroup%
  \makeatletter%
  \providecommand\color[2][]{%
    \errmessage{(Inkscape) Color is used for the text in Inkscape, but the package 'color.sty' is not loaded}%
    \renewcommand\color[2][]{}%
  }%
  \providecommand\transparent[1]{%
    \errmessage{(Inkscape) Transparency is used (non-zero) for the text in Inkscape, but the package 'transparent.sty' is not loaded}%
    \renewcommand\transparent[1]{}%
  }%
  \providecommand\rotatebox[2]{#2}%
  \ifx\svgwidth\undefined%
    \setlength{\unitlength}{102.83356934bp}%
    \ifx\svgscale\undefined%
      \relax%
    \else%
      \setlength{\unitlength}{\unitlength * \real{\svgscale}}%
    \fi%
  \else%
    \setlength{\unitlength}{\svgwidth}%
  \fi%
  \global\let\svgwidth\undefined%
  \global\let\svgscale\undefined%
  \makeatother%
  \begin{picture}(1,0.47438058)%
    \put(0,0){\includegraphics[width=\unitlength]{Trivial_curve_rel2.pdf}}%
    \put(0.05530736,0.40761267){\color[rgb]{0,0,0}\makebox(0,0)[lb]{\smash{$n$}}}%
    \put(0.50339803,0.13650277){\color[rgb]{0,0,0}\makebox(0,0)[lb]{\smash{$=$}}}%
    \put(0.75722305,0.37688152){\color[rgb]{0,0,0}\makebox(0,0)[lb]{\smash{$n-1$}}}%
  \end{picture}%
\endgroup%

%% file: Trivial_curve_rel3.pdf_tex
\begingroup%
  \makeatletter%
  \providecommand\color[2][]{%
    \errmessage{(Inkscape) Color is used for the text in Inkscape, but the package 'color.sty' is not loaded}%
    \renewcommand\color[2][]{}%
  }%
  \providecommand\transparent[1]{%
    \errmessage{(Inkscape) Transparency is used (non-zero) for the text in Inkscape, but the package 'transparent.sty' is not loaded}%
    \renewcommand\transparent[1]{}%
  }%
  \providecommand\rotatebox[2]{#2}%
  \ifx\svgwidth\undefined%
    \setlength{\unitlength}{188.17822266bp}%
    \ifx\svgscale\undefined%
      \relax%
    \else%
      \setlength{\unitlength}{\unitlength * \real{\svgscale}}%
    \fi%
  \else%
    \setlength{\unitlength}{\svgwidth}%
  \fi%
  \global\let\svgwidth\undefined%
  \global\let\svgscale\undefined%
  \makeatother%
  \begin{picture}(1,0.23816125)%
    \put(0,0){\includegraphics[width=\unitlength]{Trivial_curve_rel3.pdf}}%
    \put(0.03111322,0.21157409){\color[rgb]{0,0,0}\makebox(0,0)[lb]{\smash{$n$}}}%
    \put(0.25657185,0.11172815){\color[rgb]{0,0,0}\makebox(0,0)[lb]{\smash{$=$}}}%
    \put(0.33093541,0.11145992){\color[rgb]{0,0,0}\makebox(0,0)[lb]{\smash{$-A^2$}}}%
    \put(0.52263568,0.16770066){\color[rgb]{0,0,0}\makebox(0,0)[lb]{\smash{$n+1$}}}%
    \put(0.69873438,0.12075387){\color[rgb]{0,0,0}\makebox(0,0)[lb]{\smash{$-A^4$}}}%
    \put(0.91667301,0.16572784){\color[rgb]{0,0,0}\makebox(0,0)[lb]{\smash{$n-1$}}}%
  \end{picture}%
\endgroup%

%% file: Equation_nsep2-nouvelle.pdf_tex
\begingroup%
  \makeatletter%
  \providecommand\color[2][]{%
    \errmessage{(Inkscape) Color is used for the text in Inkscape, but the package 'color.sty' is not loaded}%
    \renewcommand\color[2][]{}%
  }%
  \providecommand\transparent[1]{%
    \errmessage{(Inkscape) Transparency is used (non-zero) for the text in Inkscape, but the package 'transparent.sty' is not loaded}%
    \renewcommand\transparent[1]{}%
  }%
  \providecommand\rotatebox[2]{#2}%
  \newcommand*\fsize{\dimexpr\f@size pt\relax}%
  \newcommand*\lineheight[1]{\fontsize{\fsize}{#1\fsize}\selectfont}%
  \ifx\svgwidth\undefined%
    \setlength{\unitlength}{117.41639328bp}%
    \ifx\svgscale\undefined%
      \relax%
    \else%
      \setlength{\unitlength}{\unitlength * \real{\svgscale}}%
    \fi%
  \else%
    \setlength{\unitlength}{\svgwidth}%
  \fi%
  \global\let\svgwidth\undefined%
  \global\let\svgscale\undefined%
  \makeatother%
  \begin{picture}(1,0.38242357)%
    \lineheight{1}%
    \setlength\tabcolsep{0pt}%
    \put(0,0){\includegraphics[width=\unitlength,page=1]{Equation_nsep2-nouvelle.pdf}}%
    \put(0.83140429,0.31470731){\color[rgb]{0,0,0}\makebox(0,0)[lt]{\lineheight{0}\smash{\begin{tabular}[t]{l}$n+1$\end{tabular}}}}%
    \put(0.33699995,0.12632736){\color[rgb]{0,0,0}\makebox(0,0)[lt]{\lineheight{0}\smash{\begin{tabular}[t]{l}$=$\end{tabular}}}}%
    \put(0,0){\includegraphics[width=\unitlength,page=2]{Equation_nsep2-nouvelle.pdf}}%
    \put(0.10401606,0.32167675){\color[rgb]{0,0,0}\makebox(0,0)[lt]{\lineheight{0}\smash{\begin{tabular}[t]{l}$n$\end{tabular}}}}%
    \put(0,0){\includegraphics[width=\unitlength,page=3]{Equation_nsep2-nouvelle.pdf}}%
  \end{picture}%
\endgroup%

%% file: Trivial_curve_rel4.pdf_tex
\begingroup%
  \makeatletter%
  \providecommand\color[2][]{%
    \errmessage{(Inkscape) Color is used for the text in Inkscape, but the package 'color.sty' is not loaded}%
    \renewcommand\color[2][]{}%
  }%
  \providecommand\transparent[1]{%
    \errmessage{(Inkscape) Transparency is used (non-zero) for the text in Inkscape, but the package 'transparent.sty' is not loaded}%
    \renewcommand\transparent[1]{}%
  }%
  \providecommand\rotatebox[2]{#2}%
  \ifx\svgwidth\undefined%
    \setlength{\unitlength}{44.74950867bp}%
    \ifx\svgscale\undefined%
      \relax%
    \else%
      \setlength{\unitlength}{\unitlength * \real{\svgscale}}%
    \fi%
  \else%
    \setlength{\unitlength}{\svgwidth}%
  \fi%
  \global\let\svgwidth\undefined%
  \global\let\svgscale\undefined%
  \makeatother%
  \begin{picture}(1,0.44677338)%
    \put(0,0){\includegraphics[width=\unitlength]{Trivial_curve_rel4.pdf}}%
  \end{picture}%
\endgroup%

%% file: Trivial_curve_rel5.pdf_tex
\begingroup%
  \makeatletter%
  \providecommand\color[2][]{%
    \errmessage{(Inkscape) Color is used for the text in Inkscape, but the package 'color.sty' is not loaded}%
    \renewcommand\color[2][]{}%
  }%
  \providecommand\transparent[1]{%
    \errmessage{(Inkscape) Transparency is used (non-zero) for the text in Inkscape, but the package 'transparent.sty' is not loaded}%
    \renewcommand\transparent[1]{}%
  }%
  \providecommand\rotatebox[2]{#2}%
  \ifx\svgwidth\undefined%
    \setlength{\unitlength}{298.06628418bp}%
    \ifx\svgscale\undefined%
      \relax%
    \else%
      \setlength{\unitlength}{\unitlength * \real{\svgscale}}%
    \fi%
  \else%
    \setlength{\unitlength}{\svgwidth}%
  \fi%
  \global\let\svgwidth\undefined%
  \global\let\svgscale\undefined%
  \makeatother%
  \begin{picture}(1,0.1270414)%
    \put(0,0){\includegraphics[width=\unitlength]{Trivial_curve_rel5.pdf}}%
    \put(0.17817488,0.11025612){\color[rgb]{0,0,0}\makebox(0,0)[lb]{\smash{$n$}}}%
    \put(-0.00076535,0.06958789){\color[rgb]{0,0,0}\makebox(0,0)[lb]{\smash{$A^{-n-2}$}}}%
    \put(0.49678389,0.10914306){\color[rgb]{0,0,0}\makebox(0,0)[lb]{\smash{$-n$}}}%
    \put(0.31784366,0.06847483){\color[rgb]{0,0,0}\makebox(0,0)[lb]{\smash{$-A^{n+2}$}}}%
    \put(0.24459968,0.00451871){\color[rgb]{0,0,0}\makebox(0,0)[lb]{\smash{$=-A^{-n}S_{n+1}-A^{-n+2}S_{n-1}+A^{n+4}S_{-n+1}+A^{n+6}S_{-n-1}$}}}%
  \end{picture}%
\endgroup%

%% file: Lickorish.pdf_tex
\begingroup%
  \makeatletter%
  \providecommand\color[2][]{%
    \errmessage{(Inkscape) Color is used for the text in Inkscape, but the package 'color.sty' is not loaded}%
    \renewcommand\color[2][]{}%
  }%
  \providecommand\transparent[1]{%
    \errmessage{(Inkscape) Transparency is used (non-zero) for the text in Inkscape, but the package 'transparent.sty' is not loaded}%
    \renewcommand\transparent[1]{}%
  }%
  \providecommand\rotatebox[2]{#2}%
  \ifx\svgwidth\undefined%
    \setlength{\unitlength}{253.1394043bp}%
    \ifx\svgscale\undefined%
      \relax%
    \else%
      \setlength{\unitlength}{\unitlength * \real{\svgscale}}%
    \fi%
  \else%
    \setlength{\unitlength}{\svgwidth}%
  \fi%
  \global\let\svgwidth\undefined%
  \global\let\svgscale\undefined%
  \makeatother%
  \begin{picture}(1,0.44360159)%
    \put(0,0){\includegraphics[width=\unitlength]{Lickorish.pdf}}%
    \put(0.1033834,0.22110446){\color[rgb]{1,0,0}\makebox(0,0)[lb]{\smash{$a_1$}}}%
    \put(0.14577577,0.3279536){\color[rgb]{1,0,0}\makebox(0,0)[lb]{\smash{$b_1$}}}%
    \put(0.25427658,0.38983148){\color[rgb]{1,0,0}\makebox(0,0)[lb]{\smash{$a_2$}}}%
    \put(0.36277743,0.24511169){\color[rgb]{1,0,0}\makebox(0,0)[lb]{\smash{$b_2$}}}%
    \put(0.3408401,0.14299638){\color[rgb]{1,0,0}\makebox(0,0)[lb]{\smash{$a_3$}}}%
    \put(0.13866098,0.16057915){\color[rgb]{1,0,0}\makebox(0,0)[lb]{\smash{$b_3$}}}%
    \put(0.5857081,0.42973081){\color[rgb]{0,0,1}\makebox(0,0)[lb]{\smash{$\alpha_1$}}}%
    \put(0.91180347,0.43175964){\color[rgb]{0,0,1}\makebox(0,0)[lb]{\smash{$\alpha_2$}}}%
    \put(0.56910688,0.26675203){\color[rgb]{0,0,1}\makebox(0,0)[lb]{\smash{$\beta_1$}}}%
    \put(0.90635398,0.25379359){\color[rgb]{0,0,1}\makebox(0,0)[lb]{\smash{$\beta_2$}}}%
    \put(0.75527774,0.00301041){\color[rgb]{0,0,1}\makebox(0,0)[lb]{\smash{$\alpha_3$}}}%
    \put(0.82049675,0.10106821){\color[rgb]{0,0,1}\makebox(0,0)[lb]{\smash{$\beta_3$}}}%
    \put(0.74223383,0.3722487){\color[rgb]{0,0,1}\makebox(0,0)[lb]{\smash{$\gamma_1$}}}%
    \put(0.84934465,0.16639162){\color[rgb]{0,0,1}\makebox(0,0)[lb]{\smash{$\gamma_2$}}}%
    \put(0.60477258,0.22968521){\color[rgb]{0,0,1}\makebox(0,0)[lb]{\smash{$\gamma_3$}}}%
  \end{picture}%
\endgroup%

%% file: Relation_twist1.pdf_tex
\begingroup%
  \makeatletter%
  \providecommand\color[2][]{%
    \errmessage{(Inkscape) Color is used for the text in Inkscape, but the package 'color.sty' is not loaded}%
    \renewcommand\color[2][]{}%
  }%
  \providecommand\transparent[1]{%
    \errmessage{(Inkscape) Transparency is used (non-zero) for the text in Inkscape, but the package 'transparent.sty' is not loaded}%
    \renewcommand\transparent[1]{}%
  }%
  \providecommand\rotatebox[2]{#2}%
  \ifx\svgwidth\undefined%
    \setlength{\unitlength}{300.91520996bp}%
    \ifx\svgscale\undefined%
      \relax%
    \else%
      \setlength{\unitlength}{\unitlength * \real{\svgscale}}%
    \fi%
  \else%
    \setlength{\unitlength}{\svgwidth}%
  \fi%
  \global\let\svgwidth\undefined%
  \global\let\svgscale\undefined%
  \makeatother%
  \begin{picture}(1,0.42513935)%
    \put(0,0){\includegraphics[width=\unitlength]{Relation_twist1.pdf}}%
    \put(0.23070331,0.33227611){\color[rgb]{0,0,0}\makebox(0,0)[lb]{\smash{$=$}}}%
    \put(-0.00037064,0.09990345){\color[rgb]{0,0,0}\makebox(0,0)[lb]{\smash{$A$}}}%
    \put(0.21231386,0.09990345){\color[rgb]{0,0,0}\makebox(0,0)[lb]{\smash{$+A^{-1}$}}}%
    \put(0.48602131,0.09156852){\color[rgb]{0,0,0}\makebox(0,0)[lb]{\smash{$= A$}}}%
    \put(0.7335428,0.09548271){\color[rgb]{0,0,0}\makebox(0,0)[lb]{\smash{$+A^{-1}$}}}%
    \put(0.44303857,0.32872917){\color[rgb]{0,0,0}\makebox(0,0)[lb]{\smash{, hence:}}}%
  \end{picture}%
\endgroup%

%% file: Relation_twist2.pdf_tex
\begingroup%
  \makeatletter%
  \providecommand\color[2][]{%
    \errmessage{(Inkscape) Color is used for the text in Inkscape, but the package 'color.sty' is not loaded}%
    \renewcommand\color[2][]{}%
  }%
  \providecommand\transparent[1]{%
    \errmessage{(Inkscape) Transparency is used (non-zero) for the text in Inkscape, but the package 'transparent.sty' is not loaded}%
    \renewcommand\transparent[1]{}%
  }%
  \providecommand\rotatebox[2]{#2}%
  \ifx\svgwidth\undefined%
    \setlength{\unitlength}{389.8668457bp}%
    \ifx\svgscale\undefined%
      \relax%
    \else%
      \setlength{\unitlength}{\unitlength * \real{\svgscale}}%
    \fi%
  \else%
    \setlength{\unitlength}{\svgwidth}%
  \fi%
  \global\let\svgwidth\undefined%
  \global\let\svgscale\undefined%
  \makeatother%
  \begin{picture}(1,0.38919566)%
    \put(0,0){\includegraphics[width=\unitlength]{Relation_twist2.pdf}}%
    \put(0.17874858,0.32737586){\color[rgb]{0,0,0}\makebox(0,0)[lb]{\smash{$=$}}}%
    \put(-0.00050935,0.19117837){\color[rgb]{0,0,0}\makebox(0,0)[lb]{\smash{$A^2$}}}%
    \put(0.21057585,0.19195222){\color[rgb]{0,0,0}\makebox(0,0)[lb]{\smash{$+$}}}%
    \put(0.39812685,0.19195222){\color[rgb]{0,0,0}\makebox(0,0)[lb]{\smash{$+$}}}%
    \put(0.60035018,0.19195222){\color[rgb]{0,0,0}\makebox(0,0)[lb]{\smash{$+A^{-2}$}}}%
    \put(0.12176036,0.05327701){\color[rgb]{0,0,0}\makebox(0,0)[lb]{\smash{$=A^{-2}$}}}%
    \put(0.37686933,0.0540509){\color[rgb]{0,0,0}\makebox(0,0)[lb]{\smash{$+$}}}%
    \put(0.56442029,0.0540509){\color[rgb]{0,0,0}\makebox(0,0)[lb]{\smash{$+$}}}%
    \put(0.76664376,0.0540509){\color[rgb]{0,0,0}\makebox(0,0)[lb]{\smash{$+A^2$}}}%
    \put(0.36395096,0.32560953){\color[rgb]{0,0,0}\makebox(0,0)[lb]{\smash{, thus:}}}%
  \end{picture}%
\endgroup%

%% file: Extra_curves.pdf_tex
\begingroup%
  \makeatletter%
  \providecommand\color[2][]{%
    \errmessage{(Inkscape) Color is used for the text in Inkscape, but the package 'color.sty' is not loaded}%
    \renewcommand\color[2][]{}%
  }%
  \providecommand\transparent[1]{%
    \errmessage{(Inkscape) Transparency is used (non-zero) for the text in Inkscape, but the package 'transparent.sty' is not loaded}%
    \renewcommand\transparent[1]{}%
  }%
  \providecommand\rotatebox[2]{#2}%
  \ifx\svgwidth\undefined%
    \setlength{\unitlength}{195.03532715bp}%
    \ifx\svgscale\undefined%
      \relax%
    \else%
      \setlength{\unitlength}{\unitlength * \real{\svgscale}}%
    \fi%
  \else%
    \setlength{\unitlength}{\svgwidth}%
  \fi%
  \global\let\svgwidth\undefined%
  \global\let\svgscale\undefined%
  \makeatother%
  \begin{picture}(1,0.20479477)%
    \put(0,0){\includegraphics[width=\unitlength]{Extra_curves.pdf}}%
    \put(0.12223187,0.18527521){\color[rgb]{0,1,0}\makebox(0,0)[lb]{\smash{$A_i$}}}%
    \put(0.12154182,0.04418442){\color[rgb]{1,0,0}\makebox(0,0)[lb]{\smash{$B_i$}}}%
    \put(0.39618779,0.18035213){\color[rgb]{0,1,0}\makebox(0,0)[lb]{\smash{$A_i$}}}%
    \put(0.39549779,0.03926136){\color[rgb]{1,0,0}\makebox(0,0)[lb]{\smash{$B_i'$}}}%
    \put(0.66552998,0.17812273){\color[rgb]{0,1,0}\makebox(0,0)[lb]{\smash{$C_i$}}}%
    \put(0.66483995,0.03703194){\color[rgb]{1,0,0}\makebox(0,0)[lb]{\smash{$D_i$}}}%
    \put(0.9082661,0.17947589){\color[rgb]{0,1,0}\makebox(0,0)[lb]{\smash{$C_i$}}}%
    \put(0.89170489,0.04074974){\color[rgb]{1,0,0}\makebox(0,0)[lb]{\smash{$D_i'$}}}%
  \end{picture}%
\endgroup%